\documentclass[11pt]{amsart}
\usepackage[latin9]{inputenc}
\usepackage[letterpaper]{geometry}
\usepackage{amstext}
\usepackage{amsthm}
\usepackage{amssymb}

\makeatletter

\providecommand{\tabularnewline}{\\}

\numberwithin{equation}{section}
\numberwithin{figure}{section}
\theoremstyle{plain}
\newtheorem{thm}{\protect\theoremname}
\theoremstyle{plain}
\newtheorem{prop}[thm]{\protect\propositionname}
\theoremstyle{plain}
\newtheorem{cor}[thm]{\protect\corollaryname}
\theoremstyle{remark}
\newtheorem{rem}[thm]{\protect\remarkname}

\usepackage{graphicx}
\usepackage{epstopdf}
\makeatother

\providecommand{\corollaryname}{Corollary}
\providecommand{\propositionname}{Proposition}
\providecommand{\remarkname}{Remark}
\providecommand{\theoremname}{Theorem}

\begin{document}
\title{Corrigendum to ``Lens spaces, isospectral on forms but not on functions''
}
\author{Ruth Gornet and Jeffrey McGowan}

\maketitle
This note means to correct an error by the authors for the composite
$q$ case in the paper {\em Lens Spaces, Isospectral on Forms but
not on Functions,} published in \textit{LMS J. Comput. Math.} 9 (2006),
270\textendash 286. We would like to thanks Sebastian Boldt and Emilio Lauret for pointing this error out.  All calculations and examples presented in \cite{GM}
for prime $q$ remain valid, and we include detailed calculations
below justifying this. Our original mistake was to conclude that Formula
(3.11) \cite[p. 399]{Ikeda} remained true for all $q$ when in fact
it is only valid if $q$ is prime. This means formulas (3) and (4)
in \cite{GM} must be reworked to account for complications when $q$
is composite. We provide below the revised formulas for the specific
composite cases of $q$ a prime power and $q$ a product of distinct
primes. Analogous formulas for any composite case can be computed
in a similar fashion.

The computations have been re-run on computers, and new $q$ prime
examples have been generated that satisfy all the claimed profcoclperties
from \cite{GM}. These are given in detail in Section 2. The computer
code is provided in an appendix. The computer code has been corrected
for the composite $q$ cases listed above, but memory storage limitations
have been reached and no composite $q$ examples of sporadically $p$-isospectral
lens spaces have yet been generated. The authors are actively working
toward generating new examples in the $q$ composite case.

We should note that there has been follow up investigation of the $p$-form spectrum for lens spaces, see for example \cite{L1,LMP}

\section{Correction and/or Justification of Generating Function Formulas}

In this section, we correct and justify in the case of $q$ composite
and justify in the case of $q$ prime formulas (3) and (4) from \cite{GM}. 

Let $\mathbb{R}^{2n}$ $\left(n\geq2\right)$ be $2n$-dimensional
Euclidean space and let $S^{2n-1}$ be the unit sphere centered at
the origin. Let $g\in O\left(2n\right)$ generate a finite, fixed-point
free cyclic subgroup $G$ of $O\left(2n\right),$ which acts on $S^{2n-1.}$

From Formula (2.13) in \cite[p. 394]{Ikeda}, the generating function
for $G$ is given by
\begin{align*}
F_{G}^{p}\left(z\right)= & \left|G\right|^{-1}\sum_{b=0}^{p}\left(-1\right)^{b}\left(z^{-b}-z^{b+2}\right)\sum_{t=1}^{q}\frac{\chi^{p-b}\left(g^{t}\right)}{\det\left(z-g^{t}\right)}+\left(-1\right)^{p+1}z^{-p}\\
= & \left|G\right|^{-1}\sum_{b=0}^{p}\left(-1\right)^{b}\left(z^{-b}-z^{b+2}\right)F^{p-b}\left(G:z\right)+\left(-1\right)^{p+1}z^{-p},
\end{align*}
where $\chi^{p}$ denotes the character of the natural representation
of $G$ on $\Lambda^{p},$ and 
\[
F^{p}\left(G:z\right)=\sum_{t=1}^{q}\frac{\chi^{p}\left(g^{t}\right)}{\det\left(z-g^{t}\right)}.
\]
We abuse notation slightly by denoting $zI_{2k}-g^{t}$ by $z-g^{t}.$ 
\begin{prop}
\cite[Prop 2.1 (1), p. 391]{Ikeda}, Two distinct cyclic subgroups $G$
and $G'$ yield lens spaces with the same $p$-form spectrum if and
only if
\[
F_{G}^{p}\left(z\right)=F_{G'}^{p}\left(z\right)
\]
and 
\[
F_{G}^{p-1}\left(z\right)=F_{G'}^{p-1}\left(z\right).
\]
\end{prop}

Let $\varPsi_{q}\left(z\right)$ denote the cylclotomic polynomial
of $q.$ So letting $\eta$ run through all primitive $q$th roots
of unity except $1,$ $\Psi_{q}\left(z\right)$ is defined as 
\begin{align*}
\varPsi_{q}\left(z\right):= & \prod_{\eta}\left(z-\eta\right).
\end{align*}

Let $q$ be a positive integer and denote 
\[
RP_{q}=\left\{ \pm t\;|\;t\in\mathbb{N},t<\frac{q}{2},\left(t,q\right)=1\right\} \subseteq\left(-\frac{q}{2},\frac{q}{2}\right),
\]
 a complete representative of integers relatively prime to $q$ that
is contained in the interval $\left(-\frac{q}{2},\frac{q}{2}\right).$
When we wish the representatives to be non-negative, we denote
\[
RP'_{q}=\left\{ t\;|\;t\in\mathbb{N},0<t<q,\left(t,q\right)=1\right\} =RP_{q}+\frac{q}{2}\subseteq\left(0,q\right),
\]
We also denote, for real valued $v,$ 
\[
\textrm{Rot}\left(\nu\right)=\left(\begin{array}{cc}
\cos\left(2\pi v\right) & \sin\left(2\pi\nu\right)\\
-\sin\left(2\pi\nu\right) & \cos\left(2\pi\nu\right)
\end{array}\right).
\]
Note that $\textrm{Rot}\left(\nu\right)$ has eigenvalues $\exp\left(\pm2\pi i\nu\right),$
and $\left(\textrm{Rot}\left(\nu\right)\right)^{t}=\textrm{Rot}\left(t\nu\right).$ 

We choose $R=\left\{ r_{1},r_{2},\ldots,r_{n}\right\} \subseteq RP_{q}$
such that $R\cap\left(-R\right)=\textrm{�}$ modulo $q,$ that is,
the set $R$ does not contain an element and its additive inverse,
modulo $q.$ Here $-R=\left\{ -r_{1},-r_{2},\ldots,-r_{n}\right\} .$
The set $R$ determines $g\in O\left(2n\right)$ by 
\[
g=\left(\begin{array}{cccc}
\textrm{Rot}\left(r_{1}/q\right) & 0 & \cdots & 0\\
0 & \textrm{Rot}\left(r_{2}/q\right) & \cdots & 0\\
\vdots & \vdots & \ddots & 0\\
0 & 0 & \cdots & \textrm{Rot}\left(r_{n}/q\right)
\end{array}\right).
\]
Once $n$ is chosen, we define $k$ by 
\[
\Phi\left(q\right)=2n+2k,
\]
where $\Phi\left(q\right)$ is the Euler $\Phi$-function of $q,$
also known as the Euler totient function of $q.$ Euler's totient
function counts the positive integers up to $q$ that are relatively
prime to $q.$ Recall that $\Phi\left(q\right)$ is even if $q\geq3.$ 

Once $R$ is chosen, we then choose $S=\left\{ s_{1},s_{2},\ldots,s_{k}\right\} \subseteq RP-\pm R$
such that $S\cap\left(-S\right)=\textrm{�}$ modulo $q.$ The set
$S$ determines $\bar{g}\in O\left(2k\right)$by 
\[
\bar{g}=\left(\begin{array}{cccc}
\textrm{Rot}\left(s_{1}/q\right) & 0 & \cdots & 0\\
0 & \textrm{Rot}\left(s_{2}/q\right) & \cdots & 0\\
\vdots & \vdots & \ddots & 0\\
0 & 0 & \cdots & \textrm{Rot}\left(s_{k}/q\right)
\end{array}\right).
\]
We denote $\pm R=\left\{ \pm r_{1},\pm r_{2},\ldots,\pm r_{n}\right\} ,$
which contains exactly $2n$ elements modulo $q,$ and $\pm S=\left\{ \pm s_{1},\pm s_{2},\ldots,\pm s_{k}\right\} ,$
which contains exactly $2k$ elements modulo $q.$ With these assumptions,
$RP_{q}=\pm R\cup\pm S,$ with $\pm R$ and $\pm S$ disjoint. 

From the relation between determinant and the natural representation
of $O\left(2n\right)$ on $\Lambda^{p}\left(\mathbb{R}^{2n}\right)$,
it is standard to compute (see \cite[(2.16)]{Ikeda}, with corrected
exponent of $-1$ ) 
\begin{equation}
\det\left(z-g^{t}\right)=\sum_{a=0}^{2n}\left(-1\right)^{a}\chi^{a}\left(g^{t}\right)z^{a}\label{eq:DET}
\end{equation}
and 
\begin{equation}
\det\left(z-\bar{g}^{t}\right)=\sum_{a=0}^{2k}\left(-1\right)^{a}\chi^{a}\left(\bar{g}^{t}\right)z^{a}.\label{eq:DET-bar}
\end{equation}

Let $\varPsi_{q}\left(z\right)$ denote the cylclotomic polynomial
of $q.$ Letting $\eta$ run through all primitive $q$th roots of
unity except $1,$ $\Psi_{q}\left(z\right)$ is defined as 
\begin{align*}
\varPsi_{q}\left(z\right):= & \prod_{\eta}\left(z-\eta\right)
\end{align*}

Also note that 
\[
\chi^{p}\left(I_{2n}\right)=\left(\begin{array}{c}
2n\\
p
\end{array}\right).
\]

\begin{prop}
With the above notation, we have
\[
\begin{cases}
\det\left(z-g^{t}\right)\det\left(z-\bar{g}^{t}\right)=\varPsi_{q}\left(z\right), & \textrm{if }\left(t,q\right)=1\\
\det\left(z-g^{t}\right)\det\left(z-\bar{g}^{t}\right)=\left(\varPsi_{q_{j}}\left(z\right)\right)^{q_{i}-1}, & \textrm{if }q=q_{1}q_{2},\left(t,q_{i}\right)\neq1,q_{1}\neq q_{2}\\
\det\left(z-g^{t}\right)\det\left(z-\bar{g}^{t}\right)=\left(\varPsi_{q_{1}^{m-m'}}\left(z\right)\right)^{q_{1}^{m'}}, & \textrm{if }q=q_{1}^{m},\textrm{ and }1\neq\left(t,q_{1}\right)=q_{1}^{m'}<q_{1}^{m}
\end{cases}
\]
\end{prop}

\begin{proof}
That this is true when $\left(t,q\right)=1$ follows from the work
in Section 3 of Ikeda where it is proved for $q$ prime. The essential
observation is that the set of eiganvalues for $g$ and $\bar{g}$
present a complete list of integers prime to $q,$ and taking a power
$t$ that is prime to $q$ preserves this property for $g^{t}$and
$\bar{g}^{t}$ but possibly permutes the integers prime to $q$ between
$g^{t}$and $\bar{g}^{t}.$

Let $q=q_{1}q_{2}$ with $q_{1}$and $q_{2}$ prime and unequal. Now
assume $1\leq t\leq q-1$ and $\left(t,q_{1}\right)\neq1,$ so $t=t'q_{1}$
where $\left(t',q_{2}\right)=1$ and $1\leq t'\leq q_{2}-1.$ Then
\begin{align*}
g^{t}=g^{t'q_{1}}= & \left(\begin{array}{cccc}
\textrm{Rot}\left(t'q_{1}r_{1}/q_{1}q_{2}\right) & 0 & \cdots & 0\\
0 & \textrm{Rot}\left(t'q_{1}r_{2}/q_{1}q_{2}\right) & \cdots & 0\\
\vdots & \vdots & \ddots & 0\\
0 & 0 & \cdots & \textrm{Rot}\left(t'q_{1}r_{n}/q_{1}q_{2}\right)
\end{array}\right)\\
= & \left(\begin{array}{cccc}
\textrm{Rot}\left(t'r_{1}/q_{2}\right) & 0 & \cdots & 0\\
0 & \textrm{Rot}\left(t'r_{2}/q_{2}\right) & \cdots & 0\\
\vdots & \vdots & \ddots & 0\\
0 & 0 & \cdots & \textrm{Rot}\left(t'r_{n}/q_{2}\right)
\end{array}\right).
\end{align*}
Likewise 
\[
\bar{g}^{t}=\left(\begin{array}{cccc}
\textrm{Rot}\left(t's_{1}/q_{2}\right) & 0 & \cdots & 0\\
0 & \textrm{Rot}\left(t's_{2}/q_{2}\right) & \cdots & 0\\
\vdots & \vdots & \ddots & 0\\
0 & 0 & \cdots & \textrm{Rot}\left(t's_{k}/q_{2}\right)
\end{array}\right).
\]
Since the $r_{j}$and $s_{j}$are relatively prime to $q,$ they are
relatively prime to $q_{2}.$ Since$\pm R\cup\pm S=RP_{q}$ is a complete
set of integers prime to $q$ and contains only one of each, then
with the above assumptions, $\left\{ \pm t'r_{1},\ldots,\pm t'r_{k}\right\} \cup\left\{ \pm t's_{1},\ldots,\pm t's_{n}\right\} $
contains $q_{1}-1$ disjoint, complete sets of integers prime to $q_{2,}.$
Indeed, a complete set of integers prime to $q=q_{1}q_{2}$ can be
written 
\[
RP'_{q}=\left\{ 1,\ldots,q_{2}-1,q_{2}+1,\ldots,2q_{2}-1,\ldots,q_{2}\left(q_{1}-1\right)+1.\ldots,q_{1}q_{2}-1\right\} -\left\{ q_{1},2q_{1},\ldots,\left(q_{2}-1\right)q_{1}\right\} .
\]
Note that for $t''\in\mathbb{N},$ each subset 
\[
\left\{ \left(t''-1\right)q_{2}+1,\ldots,t''q_{2}-1\right\} 
\]
is a clearly a complete set of integers relatively prime to $q_{2}.$
Also, the set that is subtracted is a complete set of integers relatively
prime to $q_{2}.$ To see this, note that the subtracted set contains
$q_{2}-1$ elements. And let $hq_{1}=kq_{2}+r$ and $h'q_{1}=k'q_{2}+r$
with $r,h,h'$ strictly between $0$ and $q_{2}.$ Without loss of
generality, assume $h>h'.$ Then $hq_{1}-kq_{2}=h'q_{1}+k'q_{2}$
and $\left(h-h'\right)q_{1}=\left(k'-k\right)q_{2}.$ Then $0\leq h-h'<q_{2}$
and $q_{1}$and $q_{2}$relatively prime, implies $h-h$' is a multiple
of $q_{2}.$ This in turn implies $h=h'$ and hence $k=k'.$ So that
the elements of the subtracted set are $q_{2}-1$ distinct elements
modulo $q_{2},$ hence represent a complete set of integers modulo
$q_{2},$ as claimed.

In conclusion, the set $\pm R\cup\pm S$ gives $q_{1}-1$ copies of
a complete set of integers relatively prime to $q_{2}.$ Since $t'$
is relatively prime to $q_{2}$, $\left\{ \pm t'r_{1},\ldots,\pm t'r_{n}\right\} \cup\left\{ \pm t's_{1},\ldots,\pm t's_{k}\right\} $
equals $q_{1}-1$ disjoint copies of a complete set of integers relatively
prime to $q_{2}.$

The conclusion now follows for the case that $q_{1}$ and $q_{2}$
are distinct primes.

We now consider the case $q=q_{1}^{m},$ with $q_{1}$ prime and $m$
a natural number. We assume $1\leq t\leq q-1$ and $1\neq\left(t,q_{1}\right)<q^{m},$
so $t=t'q_{1}^{m'}$ where $\left(t',q_{1}\right)=1$ and $1\leq t'\leq q_{1}-1.$
So 
\begin{align*}
g^{t}=g^{t'q_{1}}= & \left(\begin{array}{cccc}
\textrm{Rot}\left(t'q_{1}^{m'}r_{1}/q_{1}^{m}\right) & 0 & \cdots & 0\\
0 & \textrm{Rot}\left(t'q_{1}^{m'}r_{2}/q_{1}^{m}\right) & \cdots & 0\\
\vdots & \vdots & \ddots & 0\\
0 & 0 & \cdots & \textrm{Rot}\left(t'q_{1}^{m'}r_{n}/q_{1}^{m}\right)
\end{array}\right)\\
= & \left(\begin{array}{cccc}
\textrm{Rot}\left(t'r_{1}/q_{1}^{m-m'}\right) & 0 & \cdots & 0\\
0 & \textrm{Rot}\left(t'r_{2}/q_{1}^{m-m'}\right) & \cdots & 0\\
\vdots & \vdots & \ddots & 0\\
0 & 0 & \cdots & \textrm{Rot}\left(t'r_{n}/q_{1}^{m-m'}\right)
\end{array}\right)
\end{align*}
likewise 
\[
\bar{g}^{t}=\left(\begin{array}{cccc}
\textrm{Rot}\left(t's_{1}/q_{1}^{m-m'}\right) & 0 & \cdots & 0\\
0 & \textrm{Rot}\left(t's_{2}/q_{1}^{m-m'}\right) & \cdots & 0\\
\vdots & \vdots & \ddots & 0\\
0 & 0 & \cdots & \textrm{Rot}\left(t's_{k}/q_{1}^{m-m'}\right)
\end{array}\right).
\]
 Since the $r_{j}$and $s_{j}$are relatively prime to $q=q_{1}^{m},$
and $q_{1}$ is prime, they are relatively prime to $q_{1}^{m-m'}.$
Since $\pm R\cup\pm S=RP_{q}$ and $t'$ is relatively prime to $q,$
$\left\{ \pm t'r_{1},\ldots,\pm t'r_{n}\right\} \cup\left\{ \pm t's_{1},\ldots,\pm t's_{k}\right\} $
equals one complete set of integers prime to $q$ but also equals
$q_{1}^{m'}$ disjoint, complete sets of integers prime to $q_{1,}^{m-m'}.$
To see this, $RP'_{q}$ can be written 
\[
RP_{q}'=\left\{ 1,\ldots,q_{1}-1,q_{1}+1,\ldots,2q_{1}-1,2q_{1}+1,\ldots,q_{1}\left(q_{1}^{m-1}-1\right)+1.\ldots,q_{1}\left(q_{1}^{m-1}-1\right)+q_{1}-2,q^{m}-1\right\} 
\]
which can be written as the disjoint union 
\[
\dot{\bigcup}_{t''=0}^{q^{m'}-1}\left(RP'_{q^{m-m''}}+t''q^{m-m''}\right)
\]
 So that $\pm R\cup\pm S$ equals $q_{1}^{m'}$ disjoint copies of
complete sets of integers relatively prime to $q_{1}^{m-m'}.$ Since
$t'$ is relatively prime to $q_{1}$, $\left\{ \pm t'r_{1},\ldots,\pm t'r_{n}\right\} \cup\left\{ \pm t's_{1},\ldots,\pm t's_{k}\right\} $
also equals $q_{1}^{m'}$ disjoint copies of complete sets of integers
relatively prime to $q_{1}^{m-m'}.$

We conclude that 
\[
\det\left(z-g^{t}\right)\det\left(z-\bar{g}^{t}\right)=\left(\varPsi_{q_{1}^{m-m'}}\left(z\right)\right)^{q_{1}^{m'}}=\left(\varPsi_{q_{1}}\left(z^{q_{1}^{m-m'-1}}\right)\right)^{q_{1}^{m'}}
\]
when $q=q_{1}^{m},\textrm{and }1\neq\left(t,q_{1}\right)=q_{1}^{m'}<q_{1}^{m}.$The
conclusion now follows for the case that $q_{1}=q_{2}$
\end{proof}
\begin{cor}
WIth notation as above, for $t\neq q$ 
\[
\begin{cases}
\frac{1}{\det\left(z-g^{t}\right)}=\frac{\det\left(z-\bar{g}^{t}\right)}{\varPsi_{q}\left(z\right)}, & \textrm{if }\left(t,q\right)=1\\
\frac{1}{\det\left(z-g^{t}\right)}=\frac{\det\left(z-\bar{g}^{t}\right)}{\left(\varPsi_{q_{j}}\left(z\right)\right)^{q_{i}-1}}, & \textrm{if }q=q_{1}q_{2}\textrm{ and }\left(t,q_{i}\right)=q_{i},\,i\neq j,q_{i}\neq q_{j}\\
\frac{1}{\det\left(z-g^{t}\right)}=\frac{\det\left(z-\bar{g}^{t}\right)}{\left(\varPsi_{q_{1}^{m-m'}}\left(z\right)\right)^{q_{1}^{m'}}}, & \textrm{if }q=q_{1}^{m},\textrm{ and }1\neq\left(t,q_{1}\right)=q_{1}^{m'}<q_{1}^{m}
\end{cases}.
\]
\end{cor}

\subsection{CASE: $q=q_{1}q_{2}$ with $q_{1},q_{2}$ distinct primes:}
\vspace{-.09 in}
\begin{align*}
F^{p}\left(G:z\right)= & \sum_{t=1}^{q}\frac{\chi^{p}\left(g^{t}\right)}{\det\left(z-g^{t}\right)}\\
= & \sum_{t=1,\left(t,q\right)=1}^{q-1}\frac{\chi^{p}\left(g^{t}\right)}{\det\left(z-g^{t}\right)}+\sum_{t=1,\left(t,q_{1}\right)\neq1}^{q-1}\frac{\chi^{p}\left(g^{t}\right)}{\det\left(z-g^{t}\right)}+\sum_{t=1,\left(t,q_{2}\right)\neq1}^{q-1}\frac{\chi^{p}\left(g^{t}\right)}{\det\left(z-g^{t}\right)}+\frac{\chi^{p}\left(g^{q}\right)}{\det\left(z-g^{q}\right)}\\
= & \varPsi_{q}\left(z\right)^{-1}\sum_{t=1,\left(t,q\right)=1}^{q-1}\chi^{p}\left(g^{t}\right)\det\left(z-\bar{g}^{t}\right)+\varPsi_{q_{2}}\left(z\right)^{-\left(q_{1}-1\right)}\sum_{t=1,\left(t,q_{1}\right)\neq1}^{q-1}\chi^{p}\left(g^{t}\right)\det\left(z-\bar{g}^{t}\right)\\
 & +\varPsi_{q_{1}}\left(z\right)^{-\left(q_{2}-1\right)}\sum_{t=1,\left(t,q_{2}\right)\neq1}^{q-1}\chi^{p}\left(g^{t}\right)\det\left(z-\bar{g}^{t}\right)+\left(\begin{array}{c}
2n\\
p
\end{array}\right)\frac{1}{\left(z-1\right)^{2n}},
\end{align*}
since $g^{q}=I_{2n}.$
Using \ref{eq:DET} and \ref{eq:DET-bar}
and denoting by {*}{*} terms that depend only on $q,q_{1},q_{2},n$
and $p,$
\begin{align*}
F^{p}\left(G:z\right)= & \varPsi_{q}\left(z\right)^{-1}\sum_{t=1,\left(t,q\right)=1}^{q-1}\chi^{p}\left(g^{t}\right)\sum_{a=0}^{2k}\left(-1\right)^{a}\chi^{a}\left(\bar{g}^{t}\right)z^{a}\\
 & +\varPsi_{q_{2}}\left(z\right)^{-\left(q_{1}-1\right)}\sum_{t=1,\left(t,q_{1}\right)\neq1}^{q-1}\chi^{p}\left(g^{t}\right)\sum_{a=0}^{2k}\left(-1\right)^{a}\chi^{a}\left(\bar{g}^{t}\right)z^{a}\\
 & +\varPsi_{q_{1}}\left(z\right)^{-\left(q_{2}-1\right)}\sum_{t=1,\left(t,q_{2}\right)\neq1}^{q-1}\chi^{p}\left(g^{t}\right)\sum_{a=0}^{2k}\left(-1\right)^{a}\chi^{a}\left(\bar{g}^{t}\right)z^{a}+**\\
= & \sum_{a=0}^{2k}\left(-1\right)^{a}z^{a}\left[\varPsi_{q}\left(z\right)^{-1}\sum_{t=1,\left(t,q\right)=1}^{q-1}\chi^{p}\left(g^{t}\right)\chi^{a}\left(\bar{g}^{t}\right)\right.\\
 & +\varPsi_{q_{2}}\left(z\right)^{-\left(q_{1}-1\right)}\sum_{t=1,\left(t,q_{1}\right)\neq1}^{q-1}\chi^{p}\left(g^{t}\right)\chi^{a}\left(\bar{g}^{t}\right)\\
 & \left.+\varPsi_{q_{1}}\left(z\right)^{-\left(q_{2}-1\right)}\sum_{t=1,\left(t,q_{2}\right)\neq1}^{q-1}\chi^{p}\left(g^{t}\right)\chi^{a}\left(\bar{g}^{t}\right)\right]+**
\end{align*}
Plugging this into the formula for $F_{G}^{p}\left(z\right)$
above, and denoting by $*$ terms that depend only on $q,q_{1},q_{2},n$
and $p,$
\begin{align*}
F_{G}^{p}\left(z\right)= & \left|G\right|^{-1}\sum_{b=0}^{p}\left(-1\right)^{b}\left(z^{-b}-z^{b+2}\right)F^{p-b}\left(G:z\right)+\left(-1\right)^{p+1}z^{-p}\\
= & \left|G\right|^{-1}\sum_{b=0}^{p}\left(-1\right)^{b}\left(z^{-b}-z^{b+2}\right)\left[\sum_{a=0}^{2k}\left(-1\right)^{a}z^{a}\left[\varPsi_{q}\left(z\right)^{-1}\sum_{t=1,\left(t,q\right)=1}^{q-1}\chi^{p-b}\left(g^{t}\right)\chi^{a}\left(\bar{g}^{t}\right)\right.\right.\\
 & \left.\left.+\varPsi_{q_{2}}\left(z\right)^{-\left(q_{1}-1\right)}\sum_{t=1,\left(t,q_{1}\right)\neq1}^{q-1}\chi^{p-b}\left(g^{t}\right)\chi^{a}\left(\bar{g}^{t}\right)\right.\right.\\
 & \left.\left.+\varPsi_{q_{1}}\left(z\right)^{-\left(q_{2}-1\right)}\sum_{t=1,\left(t,q_{2}\right)\neq1}^{q-1}\chi^{p-b}\left(g^{t}\right)\chi^{a}\left(\bar{g}^{t}\right)+**\right]\right]+\left(-1\right)^{p+1}z^{-p}\\
= & \left|G\right|^{-1}\sum_{a=0}^{2k}\sum_{b=0}^{p}\left(-1\right)^{b+a}\left(z^{a-b}-z^{a+b+2}\right)\left[\varPsi_{q}\left(z\right)^{-1}\sum_{t=1,\left(t,q\right)=1}^{q-1}\chi^{p-b}\left(g^{t}\right)\chi^{a}\left(\bar{g}^{t}\right)\right.\\
 & \left.+\varPsi_{q_{2}}\left(z\right)^{-\left(q_{1}-1\right)}\sum_{t=1,\left(t,q_{1}\right)\neq1}^{q-1}\chi^{p-b}\left(g^{t}\right)\chi^{a}\left(\bar{g}^{t}\right)\right.\\
 & \left.+\varPsi_{q_{1}}\left(z\right)^{-\left(q_{2}-1\right)}\sum_{t=1,\left(t,q_{2}\right)\neq1}^{q-1}\chi^{p-b}\left(g^{t}\right)\chi^{a}\left(\bar{g}^{t}\right)\right]+*
\end{align*}

So equality of $F_{G}^{p}\left(z\right)$ for two different choices
of $G$ (keeping $q,n$ equal) depends only on equality of the coefficient
of each power of $z$ in the Laurent polynomial(s)
\[
\varPsi_{q}\left(z\right)\varPsi_{q_{1}}\left(z\right)^{\left(q_{2}-1\right)}\varPsi_{q_{2}}\left(z\right)^{\left(q_{1}-1\right)}K_{G}^{p}\left(z\right)
\]
 where 
\begin{align*}
K_{G}^{p}\left(z\right)= & \sum_{a=0}^{2k}\sum_{b=0}^{p}\left(-1\right)^{b+a}\left(z^{a-b}-z^{a+b+2}\right)\left[\varPsi_{q}\left(z\right)^{-1}\sum_{t=1,\left(t,q\right)=1}^{q-1}\chi^{p-b}\left(g^{t}\right)\chi^{a}\left(\bar{g}^{t}\right)\right.\\
 & \left.+\varPsi_{q_{2}}\left(z\right)^{-\left(q_{1}-1\right)}\sum_{t=1,\left(t,q_{1}\right)\neq1}^{q-1}\chi^{p-b}\left(g^{t}\right)\chi^{a}\left(\bar{g}^{t}\right)\right.\\
 & \left.+\varPsi_{q_{1}}\left(z\right)^{-\left(q_{2}-1\right)}\sum_{t=1,\left(t,q_{2}\right)\neq1}^{q-1}\chi^{p-b}\left(g^{t}\right)\chi^{a}\left(\bar{g}^{t}\right)\right]
\end{align*}

So that

\begin{align*}
\varPsi_{q}\left(z\right)\varPsi_{q_{1}}\left(z\right)^{\left(q_{2}-1\right)}\varPsi_{q_{2}}\left(z\right)^{\left(q_{1}-1\right)}K_{G}^{p}\left(z\right)= & \sum_{a=0}^{2k}\sum_{b=0}^{p}\left(-1\right)^{b+a}\left(z^{a-b}-z^{a+b+2}\right)\times\\
 & \times\left[\varPsi_{q_{1}}\left(z\right)^{\left(q_{2}-1\right)}\varPsi_{q_{2}}\left(z\right)^{\left(q_{1}-1\right)}\sum_{t=1,\left(t,q\right)=1}^{q-1}\chi^{p-b}\left(g^{t}\right)\chi^{a}\left(\bar{g}^{t}\right)\right.+\\
 & +\left.\varPsi_{q}\left(z\right)\varPsi_{q_{1}}\left(z\right)^{\left(q_{2}-1\right)}\sum_{t=1,\left(t,q_{1}\right)\neq1}^{q-1}\chi^{p-b}\left(g^{t}\right)\chi^{a}\left(\bar{g}^{t}\right)\right.\\
 & +\left.\varPsi_{q}\left(z\right)\varPsi_{q_{2}}\left(z\right)^{\left(q_{1}-1\right)}\sum_{t=1,\left(t,q_{2}\right)\neq1}^{q-1}\chi^{p-b}\left(g^{t}\right)\chi^{a}\left(\bar{g}^{t}\right)\right]
\end{align*}
Rewriting to facilitate isolating powers of $z,$ we have 
\begin{align*}
 & \varPsi_{q}\left(z\right)\varPsi_{q_{1}}\left(z\right)^{\left(q_{2}-1\right)}\varPsi_{q_{2}}\left(z\right)^{\left(q_{1}-1\right)}K_{G}^{p}\left(z\right)=\\
= & \sum_{c=-p}^{2k}\left(-1\right)^{c}z^{c}\left[\varPsi_{q_{1}}\left(z\right)^{\left(q_{2}-1\right)}\varPsi_{q_{2}}\left(z\right)^{\left(q_{1}-1\right)}\sum_{a=\max\left(0,c\right)}^{\min\left(2k,c+p\right)}\sum_{t=1,\left(t,q\right)=1}^{q-1}\chi^{p-a+c}\left(g^{t}\right)\chi^{a}\left(\bar{g}^{t}\right)\right.\\
 & +\left.\varPsi_{q}\left(z\right)\varPsi_{q_{1}}\left(z\right)^{\left(q_{2}-1\right)}\sum_{a=\max\left(0,c\right)}^{\min\left(2k,c+p\right)}\sum_{t=1,\left(t,q_{1}\right)\neq1}^{q-1}\chi^{p-a+c}\left(g^{t}\right)\chi^{a}\left(\bar{g}^{t}\right)\right.\\
 & +\left.\varPsi_{q}\left(z\right)\varPsi_{q_{2}}\left(z\right)^{\left(q_{1}-1\right)}\sum_{a=\max\left(0,c\right)}^{\min\left(2k,c+p\right)}\sum_{t=1,\left(t,q_{2}\right)\neq1}^{q-1}\chi^{p-a+c}\left(g^{t}\right)\chi^{a}\left(\bar{g}^{t}\right)\right]\\
 & +\sum_{c=2}^{2k+p+2}\left(-1\right)^{c+1}z^{c}\left[\varPsi_{q_{1}}\left(z\right)^{\left(q_{2}-1\right)}\varPsi_{q_{2}}\left(z\right)^{\left(q_{1}-1\right)}\sum_{a=\max\left(0,c-2-p\right)}^{\min\left(2k,c-2\right)}\sum_{t=1,\left(t,q\right)=1}^{q-1}\chi^{p+2+a-c}\left(g^{t}\right)\chi^{a}\left(\bar{g}^{t}\right)\right.\\
 & +\left.\varPsi_{q}\left(z\right)\varPsi_{q_{1}}\left(z\right)^{\left(q_{2}-1\right)}\sum_{a=\max\left(0,c-2-p\right)}^{\min\left(2k,c-2\right)}\sum_{t=1,\left(t,q_{1}\right)\neq1}^{q-1}\chi^{p+2+a-c}\left(g^{t}\right)\chi^{a}\left(\bar{g}^{t}\right)\right.\\
 & +\left.\varPsi_{q}\left(z\right)\varPsi_{q_{2}}\left(z\right)^{\left(q_{1}-1\right)}\sum_{a=\max\left(0,c-2-p\right)}^{\min\left(2k,c-2\right)}\sum_{t=1,\left(t,q_{2}\right)\neq1}^{q-1}\chi^{p+2+a-c}\left(g^{t}\right)\chi^{a}\left(\bar{g}^{t}\right)\right]
\end{align*}
Moreover, 
\vspace{-.2 in}
\begin{align*}
 & \varPsi_{q}\left(z\right)\varPsi_{q_{1}}\left(z\right)^{\left(q_{2}-1\right)}\varPsi_{q_{2}}\left(z\right)^{\left(q_{1}-1\right)}K_{G}^{p}\left(z\right)=\\
 & \varPsi_{q_{1}}\left(z\right)^{\left(q_{2}-1\right)}\varPsi_{q_{2}}\left(z\right)^{\left(q_{1}-1\right)}\left[\sum_{c=-p}^{2k}\left(-1\right)^{c}z^{c}\sum_{a=\max\left(0,c\right)}^{\min\left(2k,c+p\right)}\sum_{t=1,\left(t,q\right)=1}^{q-1}\chi^{p-a+c}\left(g^{t}\right)\chi^{a}\left(\bar{g}^{t}\right)\right.\\
 & \left.+\sum_{c=2}^{2k+p+2}\left(-1\right)^{c+1}z^{c}\sum_{a=\max\left(0,c-2-p\right)}^{\min\left(2k,c-2\right)}\sum_{t=1,\left(t,q\right)=1}^{q-1}\chi^{p+2+a-c}\left(g^{t}\right)\chi^{a}\left(\bar{g}^{t}\right)\right]\\
 & +\varPsi_{q}\left(z\right)\varPsi_{q_{1}}\left(z\right)^{\left(q_{2}-1\right)}\left[\sum_{c=-p}^{2k}\left(-1\right)^{c}z^{c}\sum_{a=\max\left(0,c\right)}^{\min\left(2k,c+p\right)}\sum_{t=1,\left(t,q_{1}\right)\neq1}^{q-1}\chi^{p-a+c}\left(g^{t}\right)\chi^{a}\left(\bar{g}^{t}\right)\right.\\
 & \left.+\sum_{c=2}^{2k+p+2}\left(-1\right)^{c+1}z^{c}\sum_{a=\max\left(0,c-2-p\right)}^{\min\left(2k,c-2\right)}\sum_{t=1,\left(t,q_{1}\right)\neq1}^{q-1}\chi^{p+2+a-c}\left(g^{t}\right)\chi^{a}\left(\bar{g}^{t}\right)\right]\\
 & +\varPsi_{q}\left(z\right)\varPsi_{q_{2}}\left(z\right)^{\left(q_{1}-1\right)}\left[\sum_{c=-p}^{2k}\left(-1\right)^{c}z^{c}\sum_{a=\max\left(0,c\right)}^{\min\left(2k,c+p\right)}\sum_{t=1,\left(t,q_{2}\right)\neq1}^{q-1}\chi^{p-a+c}\left(g^{t}\right)\chi^{a}\left(\bar{g}^{t}\right)\right.\\
 & \left.+\sum_{c=2}^{2k+p+2}\left(-1\right)^{c+1}z^{c}\sum_{a=\max\left(0,c-2-p\right)}^{\min\left(2k,c-2\right)}\sum_{t=1,\left(t,q_{2}\right)\neq1}^{q-1}\chi^{p+2+a-c}\left(g^{t}\right)\chi^{a}\left(\bar{g}^{t}\right)\right]
\end{align*}
\vspace{-2.4mm}
We should note that our computer code compares the terms in square brackets for each $c$.  It is possible that examples of isospectrality occur that we miss.  For example, we need not check equality for all the square brackets for all values of $c$.  Additionally, we could expand the cyclotomic polynomials and then compare coefficients of powers of $z$


\subsection{CASE: $q=q_{1}^{m}$ with $q_{1}$ prime:}

\begin{align*}
F^{p}\left(G:z\right)= & \sum_{t=1}^{q}\frac{\chi^{p}\left(g^{t}\right)}{\det\left(z-g^{t}\right)}\\
= & \sum_{t=1}^{q-1}\frac{\chi^{p}\left(g^{t}\right)}{\det\left(z-g^{t}\right)}+\frac{\chi^{p}\left(g^{q}\right)}{\det\left(z-g^{q}\right)}\\
= & \overset{}{\sum_{t=1,\left(t,q\right)=1}^{q-1}\frac{\chi^{p}\left(g^{t}\right)}{\det\left(z-g^{t}\right)}}+\sum_{t=1,\left(t,q_{1}\right)\neq1}^{q-1}\frac{\chi^{p}\left(g^{t}\right)}{\det\left(z-g^{t}\right)}+\frac{\chi^{p}\left(I_{2n}\right)}{\det\left(z-I_{2n}\right)}\\
= & \varPsi_{q}\left(z\right)^{-1}\sum_{t=1,\left(t,q\right)=1}^{q-1}\chi^{p}\left(g^{t}\right)\det\left(z-\bar{g}^{t}\right)\\
 & +\sum_{m'=1}^{m-1}\varPsi_{q_{1}^{m-m'}}\left(z\right)^{-q_{1}^{m'}}\sum_{t'=1,\left(t',q_{1}\right)=1}^{q_{1}^{m-m'}}\chi^{p}\left(g^{t'q_{1}^{m'}}\right)\det\left(z-\bar{g}^{t'q_{1}^{m'}}\right)+\frac{\chi^{p}\left(I_{2n}\right)}{\det\left(z-I_{2n}\right)}
\end{align*}
\newpage{}

If $m=1$; that is, if $q=q_{1}$ with $q_{1}$ prime:

\begin{align*}
F^{p}\left(G:z\right)= & \varPsi_{q}\left(z\right)^{-1}\sum_{t=1,\left(t,q\right)=1}^{q-1}\chi^{p}\left(g^{t}\right)\det\left(z-\bar{g}^{t}\right)+\frac{\chi^{p}\left(I_{2n}\right)}{\det\left(z-I_{2n}\right)}\\
= & \varPsi_{q}\left(z\right)^{-1}\sum_{t=1}^{q}\chi^{p}\left(g^{t}\right)\det\left(z-\bar{g}^{t}\right)-\varPsi_{q}\left(z\right)^{-1}\chi^{p}\left(g^{q}\right)\det\left(z-\bar{g}^{q}\right)+\frac{\chi^{p}\left(I_{2n}\right)}{\det\left(z-I_{2n}\right)}\\
= & \varPsi_{q}\left(z\right)^{-1}\sum_{t=1}^{q}\chi^{p}\left(g^{t}\right)\sum_{a=0}^{2k}\left(-1\right)^{a}\chi^{a}\left(\bar{g}^{t}\right)z^{a}\\
 & -\varPsi_{q}\left(z\right)^{-1}\chi^{p}\left(I_{2n}\right)\det\left(z-I_{2k}\right)+\left(\begin{array}{c}
2n\\
p
\end{array}\right)\frac{1}{\left(z-1\right)^{2n}}\\
= & \varPsi_{q}\left(z\right)^{-1}\sum_{t=1}^{q}\chi^{p}\left(g^{t}\right)\sum_{a=0}^{2k}\left(-1\right)^{a}\chi^{a}\left(\bar{g}^{t}\right)z^{a}+**
\end{align*}
Plugging this into the above 
\begin{align*}
F_{G}^{p}\left(z\right)= & \left|G\right|^{-1}\sum_{b=0}^{p}\left(-1\right)^{b}\left(z^{-b}-z^{b+2}\right)F^{p-b}\left(G:z\right)+\left(-1\right)^{p+1}z^{-p}\\
= & \left|G\right|^{-1}\sum_{b=0}^{p}\left(-1\right)^{b}\left(z^{-b}-z^{b+2}\right)\left[\sum_{a=0}^{2k}\left(-1\right)^{a}z^{a}\left[\varPsi_{q}\left(z\right)^{-1}\sum_{t=1}^{q}\chi^{p-b}\left(g^{t}\right)\chi^{a}\left(\bar{g}^{t}\right)\right.\right.\\
 & \left.\left.+**\right]\right]+\left(-1\right)^{p+1}z^{-p}\\
= & \varPsi_{q}\left(z\right)^{-1}\left|G\right|^{-1}\sum_{a=0}^{2k}\sum_{b=0}^{p}\left(-1\right)^{b+a}\left(z^{a-b}-z^{a+b+2}\right)\left[\sum_{t=1}^{q}\chi^{p-b}\left(g^{t}\right)\chi^{a}\left(\bar{g}^{t}\right)\right]+*
\end{align*}

where {*}{*} and $*$ depend only on $q,n$ and $p.$

So equality of $F_{G}^{p}\left(z\right)$ for two different choices
of $G$ (keeping $q,n$ equal) depends only on equality of ${\displaystyle K_{G}^{p}\left(z\right)}$
where 
\begin{align*}
K_{G}^{p}\left(z\right)= & \sum_{a=0}^{2k}\sum_{b=0}^{p}\left(-1\right)^{b+a}\left(z^{a-b}-z^{a+b+2}\right)\left[\sum_{t=1}^{q}\chi^{p-b}\left(g^{t}\right)\chi^{a}\left(\bar{g}^{t}\right)\right]
\end{align*}

Rewriting to facilitate isolating powers of $z,$ we have
\begin{align*}
K_{G}^{p}\left(z\right)= & \sum_{c=-p}^{2k}\left(-1\right)^{c}z^{c}\left[\sum_{a=\max\left(0,c\right)}^{\min\left(2k,c+p\right)}\sum_{t=1}^{q}\chi^{p-a+c}\left(g^{t}\right)\chi^{a}\left(\bar{g}^{t}\right)\right]\\
 & +\sum_{c=2}^{2k+p+2}\left(-1\right)^{c+1}z^{c}\left[\sum_{a=\max\left(0,c-2-p\right)}^{\min\left(2k,c-2\right)}\sum_{t=1}^{q}\chi^{p+2+a-c}\left(g^{t}\right)\chi^{a}\left(\bar{g}^{t}\right)\right]
\end{align*}

\newpage If $m=2,$ that is, $q=q_{1}^{2}$ with $q_{1}$ prime: 
\begin{align*}
F^{p}\left(G:z\right)= & \varPsi_{q}\left(z\right)^{-1}\sum_{t=1,\left(t,q\right)=1}^{q-1}\chi^{p}\left(g^{t}\right)\det\left(z-\bar{g}^{t}\right)\\
 & +\sum_{m'=1}^{m-1}\varPsi_{q_{1}^{m-m'}}\left(z\right)^{-q_{1}^{m'}}\sum_{t'=1,\left(t',q_{1}\right)=1}^{q_{1}^{m-m'}}\chi^{p}\left(g^{t'q_{1}^{m'}}\right)\det\left(z-\bar{g}^{t'q_{1}^{m'}}\right)+\frac{\chi^{p}\left(I_{2n}\right)}{\det\left(z-I_{2n}\right)}\\
= & \varPsi_{q_{1}^{2}}\left(z\right)^{-1}\sum_{t=1,\left(t,q_{1}\right)=1}^{q_{1}^{2}-1}\chi^{p}\left(g^{t}\right)\det\left(z-\bar{g}^{t}\right)+\left(\begin{array}{c}
2n\\
p
\end{array}\right)\frac{1}{\left(z-1\right)^{2n}}\\
 & +\varPsi_{q_{1}}\left(z\right)^{-q_{1}}\sum_{t'=1}^{q_{1}-1}\chi^{p}\left(g^{t'q_{1}}\right)\det\left(z-\bar{g}^{t'q_{1}}\right)\\
= & \varPsi_{q_{1}}\left(z^{q_{1}}\right)^{-1}\sum_{t=1,\left(t,q_{1}\right)=1}^{q_{1}^{2}-1}\chi^{p}\left(g^{t}\right)\det\left(z-\bar{g}^{t}\right)+\varPsi_{q_{1}}\left(z\right)^{-q_{1}}\sum_{t'=1}^{q_{1}-1}\chi^{p}\left(g^{t'q_{1}}\right)\det\left(z-\bar{g}^{t'q_{1}}\right)+**
\end{align*}
Plugging this into the above and denoting by $*$ terms that depend
only on $q,q_{1},q_{2},n$ and $p,$
\begin{align*}
F_{G}^{p}\left(z\right)= & \left|G\right|^{-1}\sum_{b=0}^{p}\left(-1\right)^{b}\left(z^{-b}-z^{b+2}\right)F^{p-b}\left(G:z\right)+\left(-1\right)^{p+1}z^{-p}\\
= & \left|G\right|^{-1}\sum_{b=0}^{p}\left(-1\right)^{b}\left(z^{-b}-z^{b+2}\right)\left[\varPsi_{q_{1}}\left(z^{q_{1}}\right)^{-1}\sum_{t=1,\left(t,q_{1}\right)=1}^{q_{1}^{2}-1}\chi^{p-b}\left(g^{t}\right)\det\left(z-\bar{g}^{t}\right)\right.\\
 & \left.+\varPsi_{q_{1}}\left(z\right)^{-q_{1}}\sum_{t'=1}^{q_{1}-1}\chi^{p}\left(g^{t'q_{1}}\right)\det\left(z-\bar{g}^{t'q_{1}}\right)+**\right]+\left(-1\right)^{p+1}z^{-p}\\
= & \left|G\right|^{-1}\sum_{b=0}^{p}\left(-1\right)^{b}\left(z^{-b}-z^{b+2}\right)\left[\sum_{a=0}^{2k}\left(-1\right)^{a}z^{a}\left[\varPsi_{q_{1}}\left(z^{q_{1}}\right)^{-1}\sum_{t=1,\left(t,q\right)=1}^{q_{1}^{2}-1}\chi^{p-b}\left(g^{t}\right)\chi^{a}\left(\bar{g}^{t}\right)\right.\right.\\
 & \left.\left.+\varPsi_{q_{2}}\left(z\right)^{-q_{1}}\sum_{t'=1}^{q_{1}-1}\chi^{p-b}\left(g^{t'q_{1}}\right)\chi^{a}\left(\bar{g}^{t'q_{1}}\right)+**\right]\right]+\left(-1\right)^{p+1}z^{-p}\\
= & \left|G\right|^{-1}\sum_{a=0}^{2k}\sum_{b=0}^{p}\left(-1\right)^{b+a}\left(z^{a-b}-z^{a+b+2}\right)\left[\varPsi_{q_{1}}\left(z^{q_{1}}\right)^{-1}\sum_{t=1,\left(t,q\right)=1}^{q_{1}^{2}-1}\chi^{p-b}\left(g^{t}\right)\chi^{a}\left(\bar{g}^{t}\right)\right.\\
 & \left.+\varPsi_{q_{1}}\left(z\right)^{-q_{1}}\sum_{t'=1}^{q_{1}-1}\chi^{p-b}\left(g^{t'q_{1}}\right)\chi^{a}\left(\bar{g}^{t'q_{1}}\right)\right]+*
\end{align*}

So equality of $F_{G}^{p}\left(z\right)$ for two different choices
of $G$ (keeping $q,n$ equal) depends only on equality of (the coefficient
of each power of $z$ in the Laurent polynomial(s))
\[
\varPsi_{q_{1}}\left(z^{q_{1}}\right)\varPsi_{q_{1}}\left(z\right)^{q_{1}}K_{G}^{p}\left(z\right)
\]
 where 
\begin{align*}
K_{G}^{p}\left(z\right)= & \sum_{a=0}^{2k}\sum_{b=0}^{p}\left(-1\right)^{b+a}\left(z^{a-b}-z^{a+b+2}\right)\left[\varPsi_{q_{1}}\left(z^{q_{1}}\right)^{-1}\sum_{t=1,\left(t,q\right)=1}^{q_{1}^{2}-1}\chi^{p-b}\left(g^{t}\right)\chi^{a}\left(\bar{g}^{t}\right)\right.\\
 & \left.+\varPsi_{q_{1}}\left(z\right)^{-q_{1}}\sum_{t'=1}^{q_{1}-1}\chi^{p-b}\left(g^{t'q_{1}}\right)\chi^{a}\left(\bar{g}^{t'q_{1}}\right)\right]
\end{align*}

So that

\begin{align*}
\varPsi_{q_{1}}\left(z^{q_{1}}\right)\varPsi_{q_{1}}\left(z\right)^{q_{1}}K_{G}^{p}= & \sum_{a=0}^{2k}\sum_{b=0}^{p}\left(-1\right)^{b+a}\left(z^{a-b}-z^{a+b+2}\right)\times\\
 & \times\left[\varPsi_{q_{1}}\left(z\right)^{q_{1}}\sum_{t=1,\left(t,q\right)=1}^{q_{1}^{2}-1}\chi^{p-b}\left(g^{t}\right)\chi^{a}\left(\bar{g}^{t}\right)\right.+\\
 & +\left.\varPsi_{q_{1}}\left(z^{q_{1}}\right)\sum_{t'=1}^{q_{1}-1}\chi^{p-b}\left(g^{t'q_{1}}\right)\chi^{a}\left(\bar{g}^{t'q_{1}}\right)\right]
\end{align*}
Rewriting to facilitate isolating powers of $z,$ we have
\begin{align*}
 & \varPsi_{q_{1}}\left(z^{q_{1}}\right)\varPsi_{q_{1}}\left(z\right)^{q_{1}}K_{G}^{p}\left(z\right)=\\
 & \sum_{c=-p}^{2k}\left(-1\right)^{c}z^{c}\left[\varPsi_{q_{1}}\left(z\right)^{q_{1}}\sum_{a=\max\left(0,c\right)}^{\min\left(2k,c+p\right)}\sum_{t=1,\left(t,q\right)=1}^{q_{1}^{2}-1}\chi^{p-a+c}\left(g^{t}\right)\chi^{a}\left(\bar{g}^{t}\right)\right.\\
 & +\left.\varPsi_{q_{1}}\left(z^{q_{1}}\right)\sum_{a=\max\left(0,c\right)}^{\min\left(2k,c+p\right)}\sum_{t'=1}^{q_{1}-1}\chi^{p-a+c}\left(g^{t'q_{1}}\right)\chi^{a}\left(\bar{g}^{t'q_{1}}\right)\right]\\
 & +\sum_{c=2}^{2k+p+2}\left(-1\right)^{c+1}z^{c}\left[\varPsi_{q_{1}}\left(z\right)^{q_{1}}\sum_{a=\max\left(0,c-2-p\right)}^{\min\left(2k,c-2\right)}\sum_{t=1,\left(t,q\right)=1}^{q_{1}^{2}-1}\chi^{p+2+a-c}\left(g^{t}\right)\chi^{a}\left(\bar{g}^{t}\right)\right.\\
 & +\left.\varPsi_{q_{1}}\left(z^{q_{1}}\right)\sum_{a=\max\left(0,c-2-p\right)}^{\min\left(2k,c-2\right)}\sum_{t'=1}^{q_{1}-1}\chi^{p+2+a-c}\left(g^{t'q_{1}}\right)\chi^{a}\left(\bar{g}^{t'q_{1}}\right)\right]
\end{align*}
Moreover, 

\begin{align*}
 & \varPsi_{q_{1}}\left(z^{q_{1}}\right)\varPsi_{q_{1}}\left(z\right)^{q_{1}}K_{G}^{p}\left(z\right)=\\
 & \varPsi_{q_{1}}\left(z\right)^{q_{1}}\left[\sum_{c=-p}^{2k}\left(-1\right)^{c}z^{c}\sum_{a=\max\left(0,c\right)}^{\min\left(2k,c+p\right)}\sum_{t=1,\left(t,q\right)=1}^{q_{1}^{2}-1}\chi^{p-a+c}\left(g^{t}\right)\chi^{a}\left(\bar{g}^{t}\right)\right.\\
 & \left.+\sum_{c=2}^{2k+p+2}\left(-1\right)^{c+1}z^{c}\sum_{a=\max\left(0,c-2-p\right)}^{\min\left(2k,c-2\right)}\sum_{t=1,\left(t,q\right)=1}^{q_{1}^{2}-1}\chi^{p+2+a-c}\left(g^{t}\right)\chi^{a}\left(\bar{g}^{t}\right)\right]\\
 & +\varPsi_{q_{1}}\left(z^{q_{1}}\right)\left[\sum_{c=-p}^{2k}\left(-1\right)^{c}z^{c}\sum_{a=\max\left(0,c\right)}^{\min\left(2k,c+p\right)}\sum_{t'=1}^{q_{1}-1}\chi^{p-a+c}\left(g^{t'q_{1}}\right)\chi^{a}\left(\bar{g}^{t'q_{1}}\right)\right.\\
 & \left.+\sum_{c=2}^{2k+p+2}\left(-1\right)^{c+1}z^{c}\sum_{a=\max\left(0,c-2-p\right)}^{\min\left(2k,c-2\right)}\sum_{t'=1}^{q_{1}-1}\chi^{p+2+a-c}\left(g^{t'q_{1}}\right)\chi^{a}\left(\bar{g}^{t'q_{1}}\right)\right]
\end{align*}
 For $m=3;$ that is, $q=q_{1}^{3}$ with $q_{1}$prime: 
\begin{align*}
F^{p}\left(G:z\right)= & \varPsi_{q}\left(z\right)^{-1}\sum_{t=1,\left(t,q\right)=1}^{q-1}\chi^{p}\left(g^{t}\right)\det\left(z-\bar{g}^{t}\right)\\
 & +\sum_{m'=1}^{m-1}\varPsi_{q_{1}^{m-m'}}\left(z\right)^{-q_{1}^{m'}}\sum_{t'=1,\left(t',q_{1}\right)=1}^{q_{1}^{m-m'}}\chi^{p}\left(g^{t'q_{1}^{m'}}\right)\det\left(z-\bar{g}^{t'q_{1}^{m'}}\right)+\frac{\chi^{p}\left(I_{2n}\right)}{\det\left(z-I_{2n}\right)}\\
= & \varPsi_{q_{1}^{3}}\left(z\right)^{-1}\sum_{t=1,\left(t,q_{1}\right)=1}^{q_{1}^{3}-1}\chi^{p}\left(g^{t}\right)\det\left(z-\bar{g}^{t}\right)+\left(\begin{array}{c}
2n\\
p
\end{array}\right)\frac{1}{\left(z-1\right)^{2n}}\\
 & +\varPsi_{q_{1}^{2}}\left(z\right)^{-q_{1}}\sum_{t'=1,\left(t',q_{1}\right)=1}^{q_{1}^{2}-1}\chi^{p}\left(g^{t'q_{1}}\right)\det\left(z-\bar{g}^{t'q_{1}}\right)\\
 & +\varPsi_{q_{1}}\left(z\right)^{-q_{1}^{2}}\sum_{t'=1,\left(t',q_{1}\right)=1}^{q_{1}-1}\chi^{p}\left(g^{t'q_{1}^{2}}\right)\det\left(z-\bar{g}^{t'q_{1}^{2}}\right)
\end{align*}
Continuing, and denoting by {*}{*} terms that depend only on $q,q_{1},q_{2},n$
and $p,$
\vspace{-.2in}
\begin{align*}
F^{p}\left(G:z\right)= & \varPsi_{q_{1}}\left(z^{q_{1}^{2}}\right)^{-1}\sum_{t=1,\left(t,q_{1}\right)=1}^{q_{1}^{3}-1}\chi^{p}\left(g^{t}\right)\det\left(z-\bar{g}^{t}\right)\\
 & +\varPsi_{q_{1}}\left(z^{q_{1}}\right)^{-q_{1}}\sum_{t'=1,\left(t',q_{1}\right)=1}^{q_{1}^{2}-1}\chi^{p}\left(g^{t'q_{1}}\right)\det\left(z-\bar{g}^{t'q_{1}}\right)\\
 & +\varPsi_{q_{1}}\left(z\right)^{-q_{1}^{2}}\sum_{t'=1}^{q_{1}-1}\chi^{p}\left(g^{t'q_{1}^{2}}\right)\det\left(z-\bar{g}^{t'q_{1}^{2}}\right)+**
\end{align*}
\vspace{-.2in}
Plugging this into the above where $*$ is terms that depend
only on $q,q_{1},q_{2},n$  and $p,$
\begin{align*}
F_{G}^{p}\left(z\right)= & \left|G\right|^{-1}\sum_{b=0}^{p}\left(-1\right)^{b}\left(z^{-b}-z^{b+2}\right)F^{p-b}\left(G:z\right)+\left(-1\right)^{p+1}z^{-p}\\
= & \left|G\right|^{-1}\sum_{b=0}^{p}\left(-1\right)^{b}\left(z^{-b}-z^{b+2}\right)\left[\varPsi_{q_{1}}\left(z^{q_{1}^{2}}\right)^{-1}\sum_{t=1,\left(t,q_{1}\right)=1}^{q_{1}^{3}-1}\chi^{p-b}\left(g^{t}\right)\det\left(z-\bar{g}^{t}\right)\right.\\
 & \left.+\varPsi_{q_{1}}\left(z^{q_{1}}\right)^{-q_{1}}\sum_{t'=1,\left(t',q_{1}\right)=1}^{q_{1}^{2}-1}\chi^{p-b}\left(g^{t'q_{1}}\right)\det\left(z-\bar{g}^{t'q_{1}}\right)\right.\\
 & \left.+\varPsi_{q_{1}}\left(z\right)^{-q_{1}^{2}}\sum_{t'=1}^{q_{1}-1}\chi^{p-b}\left(g^{t'q_{1}^{2}}\right)\det\left(z-\bar{g}^{t'q_{1}^{2}}\right)+**\right]+\left(-1\right)^{p+1}z^{-p}\\
= & \left|G\right|^{-1}\sum_{b=0}^{p}\left(-1\right)^{b}\left(z^{-b}-z^{b+2}\right)\left[\sum_{a=0}^{2k}\left(-1\right)^{a}z^{a}\left[\varPsi_{q_{1}}\left(z^{q_{1}^{2}}\right)^{-1}\sum_{t=1,\left(t,q_{1}\right)=1}^{q_{1}^{3}-1}\chi^{p-b}\left(g^{t}\right)\chi^{a}\left(\bar{g}^{t}\right)\right.\right.\\
 & \left.\left.+\varPsi_{q_{1}}\left(z^{q_{1}}\right)^{-q_{1}}\sum_{t'=1,\left(t',q_{1}\right)=1}^{q_{1}^{2}-1}\chi^{p-b}\left(g^{t'q_{1}}\right)\chi^{a}\left(\bar{g}^{t'q_{1}}\right)\right.\right.\\
 & \left.\left.+\varPsi_{q_{1}}\left(z\right)^{-q_{1}^{2}}\sum_{t'=1}^{q_{1}-1}\chi^{p-b}\left(g^{t'q_{1}^{2}}\right)\chi^{a}\left(\bar{g}^{t'q_{1}^{2}}\right)+**\right]\right]+\left(-1\right)^{p+1}z^{-p}\\
= & \left|G\right|^{-1}\sum_{a=0}^{2k}\sum_{b=0}^{p}\left(-1\right)^{b+a}\left(z^{a-b}-z^{a+b+2}\right)\left[\varPsi_{q_{1}}\left(z^{q_{1}^{2}}\right)^{-1}\sum_{t=1,\left(t,q_{1}\right)=1}^{q_{1}^{3}-1}\chi^{p-b}\left(g^{t}\right)\chi^{a}\left(\bar{g}^{t}\right)\right.\\
 & \left.+\varPsi_{q_{1}}\left(z^{q_{1}}\right)^{-q_{1}}\sum_{t'=1,\left(t',q_{1}\right)=1}^{q_{1}^{2}-1}\chi^{p-b}\left(g^{t'q_{1}}\right)\chi^{a}\left(\bar{g}^{t'q_{1}}\right)\right.\\
 & \left.+\varPsi_{q_{1}}\left(z\right)^{-q_{1}^{2}}\sum_{t'=1}^{q_{1}-1}\chi^{p-b}\left(g^{t'q_{1}^{2}}\right)\chi^{a}\left(\bar{g}^{t'q_{1}^{2}}\right)\right]+*
\end{align*}

So equality of $F_{G}^{p}\left(z\right)$ for two different choices
of $G$ (keeping $q,n$ equal) depends only on equality of (the coefficient
of each power of $z$ in the Laurent polynomial(s))
\[
\varPsi_{q_{1}}\left(z^{q_{1}^{2}}\right)\varPsi_{q_{1}}\left(z^{q_{1}}\right)^{q_{1}}\varPsi_{q_{1}}\left(z\right)^{q_{1}^{2}}K_{G}^{p}\left(z\right)
\]
 where 
\begin{align*}
K_{G}^{p}\left(z\right)= & \sum_{a=0}^{2k}\sum_{b=0}^{p}\left(-1\right)^{b+a}\left(z^{a-b}-z^{a+b+2}\right)\left[\varPsi_{q_{1}}\left(z^{q_{1}^{2}}\right)^{-1}\sum_{t=1,\left(t,q_{1}\right)=1}^{q_{1}^{3}-1}\chi^{p-b}\left(g^{t}\right)\chi^{a}\left(\bar{g}^{t}\right)\right.\\
 & \left.+\varPsi_{q_{1}}\left(z^{q_{1}}\right)^{-q_{1}}\sum_{t'=1,\left(t',q_{1}\right)=1}^{q_{1}^{2}-1}\chi^{p-b}\left(g^{t'q_{1}}\right)\chi^{a}\left(\bar{g}^{t'q_{1}}\right)\right.\\
 & \left.+\varPsi_{q_{1}}\left(z\right)^{-q_{1}^{2}}\sum_{t'=1}^{q_{1}-1}\chi^{p-b}\left(g^{t'q_{1}^{2}}\right)\chi^{a}\left(\bar{g}^{t'q_{1}^{2}}\right)\right]
\end{align*}

So that

\begin{align*}
 & \varPsi_{q_{1}}\left(z^{q_{1}^{2}}\right)\varPsi_{q_{1}}\left(z^{q_{1}}\right)^{q_{1}}\varPsi_{q_{1}}\left(z\right)^{q_{1}^{2}}K_{G}^{p}\left(z\right)=\\
 & \sum_{a=0}^{2k}\sum_{b=0}^{p}\left(-1\right)^{b+a}\left(z^{a-b}-z^{a+b+2}\right)\times\\
 & \times\left[\varPsi_{q_{1}}\left(z^{q_{1}}\right)^{q_{1}}\varPsi_{q_{1}}\left(z\right)^{q_{1}^{2}}\sum_{t=1,\left(t,q_{1}\right)=1}^{q_{1}^{3}-1}\chi^{p-b}\left(g^{t}\right)\chi^{a}\left(\bar{g}^{t}\right)\right.+\\
 & \left.+\varPsi_{q_{1}}\left(z^{q_{1}^{2}}\right)\varPsi_{q_{1}}\left(z\right)^{q_{1}^{2}}\sum_{t'=1,\left(t',q_{1}\right)=1}^{q_{1}^{2}-1}\chi^{p-b}\left(g^{t'q_{1}}\right)\chi^{a}\left(\bar{g}^{t'q_{1}}\right)\right.\\
 & +\left.+\varPsi_{q_{1}}\left(z^{q_{1}^{2}}\right)\varPsi_{q_{1}}\left(z^{q_{1}}\right)^{q_{1}}\sum_{t'=1}^{q_{1}-1}\chi^{p-b}\left(g^{t'q_{1}^{2}}\right)\chi^{a}\left(\bar{g}^{t'q_{1}^{2}}\right)\right]
\end{align*}
Rewriting to facilitate isolating powers of $z,$ we have 
\vspace{-2mm}
\begin{align*}
 & \varPsi_{q_{1}}\left(z^{q_{1}^{2}}\right)\varPsi_{q_{1}}\left(z^{q_{1}}\right)^{q_{1}}\varPsi_{q_{1}}\left(z\right)^{q_{1}^{2}}K_{G}^{p}\left(z\right)=\\
 & \sum_{c=-p}^{2k}\left(-1\right)^{c}z^{c}\left[\varPsi_{q_{1}}\left(z^{q_{1}}\right)^{q_{1}}\varPsi_{q_{1}}\left(z\right)^{q_{1}^{2}}\sum_{a=\max\left(0,c\right)}^{\min\left(2k,c+p\right)}\sum_{t=1,\left(t,q_{1}\right)=1}^{q_{1}^{3}-1}\chi^{p-a+c}\left(g^{t}\right)\chi^{a}\left(\bar{g}^{t}\right)\right.\\
 & \left.+\varPsi_{q_{1}}\left(z^{q_{1}^{2}}\right)\varPsi_{q_{1}}\left(z\right)^{q_{1}^{2}}\sum_{a=\max\left(0,c\right)}^{\min\left(2k,c+p\right)}\sum_{t'=1,\left(t',q_{1}\right)=1}^{q_{1}^{2}-1}\chi^{p-a+c}\left(g^{t}\right)\chi^{a}\left(\bar{g}^{t}\right)\right.\\
 & +\left.\varPsi_{q_{1}}\left(z^{q_{1}^{2}}\right)\varPsi_{q_{1}}\left(z^{q_{1}}\right)^{q_{1}}\sum_{a=\max\left(0,c\right)}^{\min\left(2k,c+p\right)}\sum_{t'=1}^{q_{1}-1}\chi^{p-a+c}\left(g^{t'q_{1}}\right)\chi^{a}\left(\bar{g}^{t'q_{1}}\right)\right]\\
 & +\sum_{c=2}^{2k+p+2}\left(-1\right)^{c+1}z^{c}\left[\varPsi_{q_{1}}\left(z^{q_{1}}\right)^{q_{1}}\varPsi_{q_{1}}\left(z\right)^{q_{1}^{2}}\sum_{a=\max\left(0,c-2-p\right)}^{\min\left(2k,c-2\right)}\sum_{t=1,\left(t,q_{1}\right)=1}^{q_{1}^{3}-1}\chi^{p+2+a-c}\left(g^{t}\right)\chi^{a}\left(\bar{g}^{t}\right)\right.\\
 & \left.+\varPsi_{q_{1}}\left(z^{q_{1}^{2}}\right)\varPsi_{q_{1}}\left(z\right)^{q_{1}^{2}}\sum_{a=\max\left(0,c-2-p\right)}^{\min\left(2k,c-2\right)}\sum_{t'=1,\left(t',q_{1}\right)=1}^{q_{1}^{2}-1}\chi^{p+2+a-c}\left(g^{t'q_{1}}\right)\chi^{a}\left(\bar{g}^{t'q_{1}}\right)\right.\\
 & +\left.\varPsi_{q_{1}}\left(z^{q_{1}^{2}}\right)\varPsi_{q_{1}}\left(z^{q_{1}}\right)^{q_{1}}\sum_{a=\max\left(0,c-2-p\right)}^{\min\left(2k,c-2\right)}\sum_{t'=1}^{q_{1}-1}\chi^{p+2+a-c}\left(g^{t'q_{1}}\right)\chi^{a}\left(\bar{g}^{t'q_{1}}\right)\right]
\end{align*}

Moreover, 
\begin{align*}
 & \varPsi_{q_{1}}\left(z^{q_{1}^{2}}\right)\varPsi_{q_{1}}\left(z^{q_{1}}\right)^{q_{1}}\varPsi_{q_{1}}\left(z\right)^{q_{1}^{2}}K_{G}^{p}\left(z\right)=\\
 & \varPsi_{q_{1}}\left(z^{q_{1}}\right)^{q_{1}}\varPsi_{q_{1}}\left(z\right)^{q_{1}^{2}}\left[\sum_{c=-p}^{2k}\left(-1\right)^{c}z^{c}\sum_{a=\max\left(0,c\right)}^{\min\left(2k,c+p\right)}\sum_{t=1,\left(t,q_{1}\right)=1}^{q_{1}^{3}-1}\chi^{p-a+c}\left(g^{t}\right)\chi^{a}\left(\bar{g}^{t}\right)\right.\\
 & \left.+\sum_{c=2}^{2k+p+2}\left(-1\right)^{c+1}z^{c}\sum_{a=\max\left(0,c-2-p\right)}^{\min\left(2k,c-2\right)}\sum_{t=1,\left(t,q_{1}\right)=1}^{q_{1}^{3}-1}\chi^{p+2+a-c}\left(g^{t}\right)\chi^{a}\left(\bar{g}^{t}\right)\right]\\
 & +\varPsi_{q_{1}}\left(z^{q_{1}^{2}}\right)\varPsi_{q_{1}}\left(z\right)^{q_{1}^{2}}\left[\sum_{c=-p}^{2k}\left(-1\right)^{c}z^{c}\sum_{a=\max\left(0,c\right)}^{\min\left(2k,c+p\right)}\sum_{t'=1,\left(t',q_{1}\right)=1}^{q_{1}^{2}-1}\chi^{p-a+c}\left(g^{t'q_{1}}\right)\chi^{a}\left(\bar{g}^{t'q_{1}}\right)\right.\\
 & \left.+\sum_{c=2}^{2k+p+2}\left(-1\right)^{c+1}z^{c}\sum_{a=\max\left(0,c-2-p\right)}^{\min\left(2k,c-2\right)}\sum_{t'=1,\left(t',q_{1}\right)=1}^{q_{1}^{2}-1}\chi^{p+2+a-c}\left(g^{t'q_{1}}\right)\chi^{a}\left(\bar{g}^{t'q_{1}}\right)\right]\\
 & +\varPsi_{q_{1}}\left(z^{q_{1}^{2}}\right)\varPsi_{q_{1}}\left(z^{q_{1}}\right)^{q_{1}}\left[\sum_{c=-p}^{2k}\left(-1\right)^{c}z^{c}\sum_{a=\max\left(0,c\right)}^{\min\left(2k,c+p\right)}\sum_{t'=1}^{q_{1}-1}\chi^{p-a+c}\left(g^{t'q_{1}}\right)\chi^{a}\left(\bar{g}^{t'q_{1}}\right)\right.\\
 & \left.+\sum_{c=2}^{2k+p+2}\left(-1\right)^{c+1}z^{c}\sum_{a=\max\left(0,c-2-p\right)}^{\min\left(2k,c-2\right)}\sum_{t'=1}^{q_{1}-1}\chi^{p+2+a-c}\left(g^{t'q_{1}}\right)\chi^{a}\left(\bar{g}^{t'q_{1}}\right)\right]
\end{align*}
We have shown:
\begin{thm}
Equality of closed or coclosed $p$ forms for different choices of
$G,$ keeping $n,p,q$ constant, depends (iff) on equality of the
polynomial $\tilde{K}_{G}^{p}\left(z\right),$ where $\tilde{K}_{G}^{p}\left(z\right)$
is defined as:
\begin{itemize}
\item If $q=q_{1}q_{2}$ with $q_{1},q_{2}$ distinct odd primes: Letting
\begin{align*}
\tilde{K}_{G}^{p}\left(z\right)= & \varPsi_{q_{1}}\left(z\right)^{\left(q_{2}-1\right)}\varPsi_{q_{2}}\left(z\right)^{\left(q_{1}-1\right)}\left[\sum_{c=-p}^{2k}\left(-1\right)^{c}z^{c}\sum_{a=\max\left(0,c\right)}^{\min\left(2k,c+p\right)}\sum_{t=1,\left(t,q\right)=1}^{q-1}\chi^{p-a+c}\left(g^{t}\right)\chi^{a}\left(\bar{g}^{t}\right)\right.\\
 & \left.+\sum_{c=2}^{2k+p+2}\left(-1\right)^{c+1}z^{c}\sum_{a=\max\left(0,c-2-p\right)}^{\min\left(2k,c-2\right)}\sum_{t=1,\left(t,q\right)=1}^{q-1}\chi^{p+2+a-c}\left(g^{t}\right)\chi^{a}\left(\bar{g}^{t}\right)\right]\\
 & +\varPsi_{q}\left(z\right)\varPsi_{q_{1}}\left(z\right)^{\left(q_{2}-1\right)}\left[\sum_{c=-p}^{2k}\left(-1\right)^{c}z^{c}\sum_{a=\max\left(0,c\right)}^{\min\left(2k,c+p\right)}\sum_{t=1,\left(t,q_{1}\right)\neq1}^{q-1}\chi^{p-a+c}\left(g^{t}\right)\chi^{a}\left(\bar{g}^{t}\right)\right.\\
 & \left.+\sum_{c=2}^{2k+p+2}\left(-1\right)^{c+1}z^{c}\sum_{a=\max\left(0,c-2-p\right)}^{\min\left(2k,c-2\right)}\sum_{t=1,\left(t,q_{1}\right)\neq1}^{q-1}\chi^{p+2+a-c}\left(g^{t}\right)\chi^{a}\left(\bar{g}^{t}\right)\right]\\
 & +\varPsi_{q}\left(z\right)\varPsi_{q_{2}}\left(z\right)^{\left(q_{1}-1\right)}\left[\sum_{c=-p}^{2k}\left(-1\right)^{c}z^{c}\sum_{a=\max\left(0,c\right)}^{\min\left(2k,c+p\right)}\sum_{t=1,\left(t,q_{2}\right)\neq1}^{q-1}\chi^{p-a+c}\left(g^{t}\right)\chi^{a}\left(\bar{g}^{t}\right)\right.\\
 & \left.+\sum_{c=2}^{2k+p+2}\left(-1\right)^{c+1}z^{c}\sum_{a=\max\left(0,c-2-p\right)}^{\min\left(2k,c-2\right)}\sum_{t=1,\left(t,q_{2}\right)\neq1}^{q-1}\chi^{p+2+a-c}\left(g^{t}\right)\chi^{a}\left(\bar{g}^{t}\right)\right]
\end{align*}
\item If $q=q_{1}$with $q_{1}$ prime:
\begin{align*}
\tilde{K}_{G}^{p}\left(z\right)= & \sum_{c=-p}^{2k}\left(-1\right)^{c}z^{c}\left[\sum_{a=\max\left(0,c\right)}^{\min\left(2k,c+p\right)}\sum_{t=1}^{q}\chi^{p-a+c}\left(g^{t}\right)\chi^{a}\left(\bar{g}^{t}\right)\right]\\
 & +\sum_{c=2}^{2k+p+2}\left(-1\right)^{c+1}z^{c}\left[\sum_{a=\max\left(0,c-2-p\right)}^{\min\left(2k,c-2\right)}\sum_{t=1}^{q}\chi^{p+2+a-c}\left(g^{t}\right)\chi^{a}\left(\bar{g}^{t}\right)\right]
\end{align*}
\item If $q=q_{1}^{2}$with $q_{1}$ prime:
\begin{align*}
\tilde{K}_{G}^{p}\left(z\right)= & \varPsi_{q_{1}}\left(z\right)^{q_{1}}\left[\sum_{c=-p}^{2k}\left(-1\right)^{c}z^{c}\sum_{a=\max\left(0,c\right)}^{\min\left(2k,c+p\right)}\sum_{t=1,\left(t,q\right)=1}^{q_{1}^{2}-1}\chi^{p-a+c}\left(g^{t}\right)\chi^{a}\left(\bar{g}^{t}\right)\right.\\
 & \left.+\sum_{c=2}^{2k+p+2}\left(-1\right)^{c+1}z^{c}\sum_{a=\max\left(0,c-2-p\right)}^{\min\left(2k,c-2\right)}\sum_{t=1,\left(t,q\right)=1}^{q_{1}^{2}-1}\chi^{p+2+a-c}\left(g^{t}\right)\chi^{a}\left(\bar{g}^{t}\right)\right]\\
 & +\varPsi_{q_{1}}\left(z^{q_{1}}\right)\left[\sum_{c=-p}^{2k}\left(-1\right)^{c}z^{c}\sum_{a=\max\left(0,c\right)}^{\min\left(2k,c+p\right)}\sum_{t'=1}^{q_{1}-1}\chi^{p-a+c}\left(g^{t'q_{1}}\right)\chi^{a}\left(\bar{g}^{t'q_{1}}\right)\right.\\
 & \left.+\sum_{c=2}^{2k+p+2}\left(-1\right)^{c+1}z^{c}\sum_{a=\max\left(0,c-2-p\right)}^{\min\left(2k,c-2\right)}\sum_{t'=1}^{q_{1}-1}\chi^{p+2+a-c}\left(g^{t'q_{1}}\right)\chi^{a}\left(\bar{g}^{t'q_{1}}\right)\right]
\end{align*}
\item If $q=q_{1}^{3}$ with $q_{1}$ prime: 
\begin{align*}
\tilde{K}_{G}^{p}\left(z\right)= & \varPsi_{q_{1}}\left(z^{q_{1}}\right)^{q_{1}}\varPsi_{q_{1}}\left(z\right)^{q_{1}^{2}}\left[\sum_{c=-p}^{2k}\left(-1\right)^{c}z^{c}\sum_{a=\max\left(0,c\right)}^{\min\left(2k,c+p\right)}\sum_{t=1,\left(t,q_{1}\right)=1}^{q_{1}^{3}-1}\chi^{p-a+c}\left(g^{t}\right)\chi^{a}\left(\bar{g}^{t}\right)\right.\\
 & \left.+\sum_{c=2}^{2k+p+2}\left(-1\right)^{c+1}z^{c}\sum_{a=\max\left(0,c-2-p\right)}^{\min\left(2k,c-2\right)}\sum_{t=1,\left(t,q_{1}\right)=1}^{q_{1}^{3}-1}\chi^{p+2+a-c}\left(g^{t}\right)\chi^{a}\left(\bar{g}^{t}\right)\right]\\
 & +\varPsi_{q_{1}}\left(z^{q_{1}^{2}}\right)\varPsi_{q_{1}}\left(z\right)^{q_{1}^{2}}\left[\sum_{c=-p}^{2k}\left(-1\right)^{c}z^{c}\sum_{a=\max\left(0,c\right)}^{\min\left(2k,c+p\right)}\sum_{t'=1,\left(t',q_{1}\right)=1}^{q_{1}^{2}-1}\chi^{p-a+c}\left(g^{t'q_{1}}\right)\chi^{a}\left(\bar{g}^{t'q_{1}}\right)\right.\\
 & \left.+\sum_{c=2}^{2k+p+2}\left(-1\right)^{c+1}z^{c}\sum_{a=\max\left(0,c-2-p\right)}^{\min\left(2k,c-2\right)}\sum_{t'=1,\left(t',q_{1}\right)=1}^{q_{1}^{2}-1}\chi^{p+2+a-c}\left(g^{t'q_{1}}\right)\chi^{a}\left(\bar{g}^{t'q_{1}}\right)\right]\\
 & +\varPsi_{q_{1}}\left(z^{q_{1}^{2}}\right)\varPsi_{q_{1}}\left(z^{q_{1}}\right)^{q_{1}}\left[\sum_{c=-p}^{2k}\left(-1\right)^{c}z^{c}\sum_{a=\max\left(0,c\right)}^{\min\left(2k,c+p\right)}\sum_{t'=1}^{q_{1}-1}\chi^{p-a+c}\left(g^{t'q_{1}}\right)\chi^{a}\left(\bar{g}^{t'q_{1}}\right)\right.\\
 & \left.+\sum_{c=2}^{2k+p+2}\left(-1\right)^{c+1}z^{c}\sum_{a=\max\left(0,c-2-p\right)}^{\min\left(2k,c-2\right)}\sum_{t'=1}^{q_{1}-1}\chi^{p+2+a-c}\left(g^{t'q_{1}}\right)\chi^{a}\left(\bar{g}^{t'q_{1}}\right)\right]
\end{align*}
\newpage{}
\end{itemize}
\end{thm}

We now rewrite the previous Theorem in terms of subset sums.
\begin{rem}
We denote the various cases as follows.
\begin{itemize}
\item CASE $q=q_{1}q_{2}:$ We need to evaluate $\tilde{C}_{G}^{\alpha,\beta,*}=\sum_{t=1,*}^{q-1}\chi^{\beta}\left(g^{t}\right)\chi^{\alpha}\left(\bar{g}^{t}\right),$
where by $*$ we mean one of the three conditions. $\left(t,q\right)=1,$
or $\left(t,q_{1}\right)\neq1$ or $\left(t,q_{2}\right)\neq1.$ 
\item CASE $q=q_{1}:$ We need to evaluate $\tilde{C}_{G}^{\alpha,\beta}=\sum_{t=1}^{q}\chi^{\beta}\left(g^{t}\right)\chi^{\alpha}\left(\bar{g}^{t}\right).$
\item CASE $q=q_{1}^{m},m>1$: We need to evaluate $\tilde{C}_{G}^{\alpha,\beta,*}=\sum_{*}\chi^{\beta}\left(g^{t}\right)\chi^{\alpha}\left(\bar{g}^{t}\right),$
where by $*$ we mean summing $t'$ from $1$ to $q=q_{1}^{m'}-1$
with$\left(t',q\right)=1$ and $t=t'q_{1}^{m-m'}$, $m'=1,\ldots,m$
\end{itemize}
\end{rem}

Before proceding, we review the character of the action of $O\left(2n\right)$
on $\varLambda^{p}\left(\mathbb{R}^{2n}\right).$ This action is defined
by 
\[
A\cdot\left(v_{1}\wedge v_{2}\wedge\cdots\wedge v_{p}\right)=Av_{1}\wedge Av_{2}\wedge\cdots\wedge Av_{p}
\]
for all $A$ in $O\left(2n\right)$ and all basis vectors $v_{1}\wedge v_{2}\wedge\cdots\wedge v_{p}$
in $\varLambda^{p}\left(\mathbb{R}^{2n}\right).$

Using the fact that the character is just the trace of the action
which is just the sum of the eigenvalues, and knowing the explicit
formulation of the eigenvalues of $g$, a straightforward calculation
shows that 
\[
\chi^{\alpha}\left(\bar{g}^{t}\right)=\sum_{\bar{\sigma}}\exp\left(\frac{2\pi it}{q}\sum\bar{\sigma}\right)
\]
where $\bar{\sigma}$ runs over the $\left(\begin{array}{c}
2k\\
2k-\alpha
\end{array}\right)$ choices of subsets of order $2k-\alpha$ of the set 
\[
RPS_{q}:=RP_{q}-\left(\pm S\right).
\]
Likewise 
\[
\chi^{\beta}\left(g^{t}\right)=\sum_{\sigma}\exp\left(\frac{2\pi it}{q}\sum\sigma\right)
\]
where $\sigma$ runs over the $\left(\begin{array}{c}
2n\\
2n-\beta
\end{array}\right)$ choices of subsets of order $2n-\beta$ of the set 
\[
RPR_{q}:=RP_{q}-\left(\pm R\right).
\]

\begin{rem}
The argument above is fairly straightforward. A simple internet search
on ``symmetric and exterior power of representation'' yields a quite nice one. that argument I found.)
\end{rem}

\newpage Note that 
\begin{align*}
\tilde{C}_{G}^{\alpha,\beta,*}= & \sum_{*}\chi^{\beta}\left(g^{t}\right)\chi^{\alpha}\left(\bar{g}^{t}\right)=\sum_{*}\left(\sum_{\bar{\sigma}}\exp\left(\frac{2\pi it}{q}\sum\bar{\sigma}\right)\right)\left(\sum_{\sigma}\exp\left(\frac{2\pi it}{q}\sum\sigma\right)\right)\\
 & =\sum_{\sigma,\bar{\sigma}}\sum_{*}\exp\left(\frac{2\pi it}{q}\sum\left(\bar{\sigma}+\sigma\right)\right),
\end{align*}
where $*$ depends on the case in point.

For the case where $q=q_{1},$ with $q_{1}$an odd prime, we have
\begin{align*}
\tilde{C}_{G}^{\alpha,\beta}= & \sum_{t=1}^{q}\chi^{\beta}\left(g^{t}\right)\chi^{\alpha}\left(\bar{g}^{t}\right)=\sum_{t=1}^{q}\left(\sum_{\bar{\sigma}}\exp\left(\frac{2\pi it}{q}\sum\bar{\sigma}\right)\right)\left(\sum_{\sigma}\exp\left(\frac{2\pi it}{q}\sum\sigma\right)\right)\\
 & =\sum_{\sigma,\bar{\sigma}}\sum_{t=1}^{q}\exp\left(\frac{2\pi it}{q}\sum\left(\bar{\sigma}+\sigma\right)\right)=\textrm{subset sum}
\end{align*}
For $\ell\in\mathbb{Z},$ 
\[
\begin{cases}
\sum_{t=1}^{q}\exp\left(\frac{2\pi it}{q}\ell\right)=0 & \textrm{if }\left(\ell,q\right)=1,\\
\sum_{t=1}^{q}\exp\left(\frac{2\pi it}{q}\ell\right)=q & \textrm{if }\ell\in q\mathbb{Z}
\end{cases}.
\]
The first line follows from the fact that the sum of all $q$ of the
$qth$ roots of unity is $0.$ The second line follow from the fact
that $\exp\left(2\pi i\mathbb{Z}\right)\equiv1.$ We define 
\begin{align*}
C_{G}^{\alpha,\beta}= & \sum_{t=1}^{q}\chi^{\beta}\left(g^{t}\right)\chi^{\alpha}\left(\bar{g}^{t}\right)\\
= & q\left(\textrm{number of pairs of sets \ensuremath{\sigma,\bar{\sigma}}such that \ensuremath{\sum_{\sigma,\bar{\sigma}}\left(\bar{\sigma}+\sigma\right)\in q\mathbb{Z}}}\right)
\end{align*}
and
\begin{align*}
H_{G}^{p}\left(z\right)=\tilde{K}_{G}^{p}\left(z\right)= & \sum_{c=-p}^{2k}\left(-1\right)^{c}z^{c}\left[\sum_{a=\max\left(0,c\right)}^{\min\left(2k,c+p\right)}\sum_{t=1}^{q}\chi^{p-a+c}\left(g^{t}\right)\chi^{a}\left(\bar{g}^{t}\right)\right]\\
 & +\sum_{c=2}^{2k+p+2}\left(-1\right)^{c+1}z^{c}\left[\sum_{a=\max\left(0,c-2-p\right)}^{\min\left(2k,c-2\right)}\sum_{t=1}^{q}\chi^{p+2+a-c}\left(g^{t}\right)\chi^{a}\left(\bar{g}^{t}\right)\right]\\
= & \sum_{c=-p}^{2k}\left(-1\right)^{c}z^{c}\left[\sum_{a=\max\left(0,c\right)}^{\min\left(2k,c+p\right)}C_{G}^{a,p-a+c}\right]\\
 & +\sum_{c=2}^{2k+p+2}\left(-1\right)^{c+1}z^{c}\left[\sum_{a=\max\left(0,c-2-p\right)}^{\min\left(2k,c-2\right)}C_{G}^{a,p+2+a-c}\right]
\end{align*}
where $\bar{\sigma}$ runs over the $\left(\begin{array}{c}
2k\\
2k-\alpha
\end{array}\right)$ choices of subsets of order $2k-\alpha$ of the set 
\[
RPS_{q}:=RP_{q}-\left(\pm S\right)
\]
 and $\sigma$ runs over the $\left(\begin{array}{c}
2n\\
2n-\beta
\end{array}\right)$ choices of subsets of order $2n-\beta$ of the set 
\[
RPR_{q}:=RP_{q}-\left(\pm R\right).
\]

\begin{rem}
\end{rem}

Let $q=q_{1}q_{2,}$ where $q_{1}$ and $q_{2}$ are distinct odd
primes. 
\begin{align*}
\tilde{C}_{G}^{\alpha,\beta,q}= & \sum_{\sigma,\bar{\sigma}}\sum_{t=1,\left(t,q\right)=1\in}^{q-1}\exp\left(\frac{2\pi it}{q}\sum\left(\bar{\sigma}+\sigma\right)\right)\\
 & =\sum_{\sigma,\bar{\sigma}}\left[\sum_{t=1}^{q}\exp\left(\frac{2\pi it}{q}\sum\left(\bar{\sigma}+\sigma\right)\right)-\sum_{t=1,\left(t,q_{1}\right)\neq1}^{q-1}\exp\left(\frac{2\pi it}{q}\sum\left(\bar{\sigma}+\sigma\right)\right)\right.\\
 & =\left.\qquad-\sum_{t=1,\left(t,q_{2}\right)\neq1}^{q-1}\exp\left(\frac{2\pi it}{q}\sum\left(\bar{\sigma}+\sigma\right)\right)-1\right]\\
 & =\sum_{\sigma,\bar{\sigma}}\left[\sum_{t=1}^{q}\exp\left(\frac{2\pi it}{q}\sum\left(\bar{\sigma}+\sigma\right)\right)\right]-\tilde{C}_{G}^{\alpha,\beta,q_{1}}-\tilde{C}_{G}^{\alpha,\beta,q_{2}}-\left(\begin{array}{c}
2k\\
2k-\alpha
\end{array}\right)\left(\begin{array}{c}
2n\\
2n-\beta
\end{array}\right)\\
\tilde{C}_{G}^{\alpha,\beta,q_{1}}= & \sum_{\sigma,\bar{\sigma}}\sum_{t=1,\left(t,q_{1}\right)\neq1,}^{q-1}\exp\left(\frac{2\pi it}{q}\sum\left(\bar{\sigma}+\sigma\right)\right)\\
 & =\sum_{\sigma,\bar{\sigma}}\left[\sum_{t=1}^{q_{2}}\exp\left(\frac{2\pi itq_{1}}{q_{1}q_{2}}\sum\left(\bar{\sigma}+\sigma\right)\right)-1\right]\\
 & =\sum_{\sigma,\bar{\sigma}}\sum_{t=1}^{q_{2}}\exp\left(\frac{2\pi it}{q_{2}}\sum\left(\bar{\sigma}+\sigma\right)\right)-\left(\begin{array}{c}
2k\\
2k-\alpha
\end{array}\right)\left(\begin{array}{c}
2n\\
2n-\beta
\end{array}\right)\\
\tilde{C}_{G}^{\alpha,\beta,q_{2}}= & \sum_{\sigma,\bar{\sigma}}\sum_{t=1,\left(t,q_{2}\right)\neq1}^{q-1}\exp\left(\frac{2\pi it}{q}\sum\left(\bar{\sigma}+\sigma\right)\right)\\
 & =\sum_{\sigma,\bar{\sigma}}\sum_{t=1}^{q_{1}}\exp\left(\frac{2\pi it}{q_{1}}\sum\left(\bar{\sigma}+\sigma\right)\right)-\left(\begin{array}{c}
2k\\
2k-\alpha
\end{array}\right)\left(\begin{array}{c}
2n\\
2n-\beta
\end{array}\right)
\end{align*}
Note that the +1 comes from the fact that t=q isn't included in the
first line, but it is included in all 3 sums in the second line.

For $\ell\in\mathbb{Z},$ 
\[
\begin{cases}
\sum_{t=1}^{q}\exp\left(\frac{2\pi it}{q}\ell\right)=0 & \textrm{if }\left(\ell,q\right)=1,\\
\sum_{t=1}^{q}\exp\left(\frac{2\pi it}{q}\ell\right)=0 & \textrm{if }\left(\ell,q_{1}\right)\neq1\textrm{ and }\left(\ell,q_{2}\right)=1\\
\sum_{t=1}^{q}\exp\left(\frac{2\pi it}{q}\ell\right)=0 & \textrm{if }\left(\ell,q_{2}\right)\neq1\textrm{ and }\left(\ell,q_{1}\right)=1\\
\sum_{t=1}^{q}\exp\left(\frac{2\pi it}{q}\ell\right)=q & \textrm{if }\ell\in q_{1}q_{2}\mathbb{Z}
\end{cases}.
\]
The first line follows from the fact that the sum of all $q$ of the
$qth$ roots of unity is $0.$ The second line follow form the fact
that $\ell=q_{1}\ell'$ and still $\left(\ell',q_{2}\right)=1.$ Now
$\sum_{t=1}^{q_{1}q_{2}}\exp\left(\frac{2\pi it}{q_{2}}\ell'\right)=0$
since it is just $q_{1}$ sums of the $q_{2}^{th}$ roots of unity.
The third line is analogous to the second line. The last line follow
from the fact that $\exp\left(2\pi it\right)=1$ for all integers
$t.$ We thus define 
\begin{align*}
C_{G}^{\alpha,\beta,q_{i}}= & q_{j}\left(\textrm{number of pairs of sets \ensuremath{\sigma,\bar{\sigma}}such that \ensuremath{\sum_{\sigma,\bar{\sigma}}\left(\bar{\sigma}+\sigma\right)\in q_{j}\mathbb{Z}}}\right),i\neq j\\
C_{G}^{\alpha,\beta,q}= & q\left(\textrm{number of pairs of sets \ensuremath{\sigma,\bar{\sigma}}such that \ensuremath{\sum_{\sigma,\bar{\sigma}}\left(\bar{\sigma}+\sigma\right)\in q\mathbb{Z}}}\right)
\end{align*}
where $\bar{\sigma}$ runs over the $\left(\begin{array}{c}
2k\\
2k-\alpha
\end{array}\right)$ choices of subsets of order $2k-\alpha$ of the set 
\[
RPS_{q}:=RP_{q}-\left(\pm S\right)
\]
 and $\sigma$ runs over the $\left(\begin{array}{c}
2n\\
2n-\beta
\end{array}\right)$ choices of subsets of order $2n-\beta$ of the set 
\[
RPR_{q}:=RP_{q}-\left(\pm R\right).
\]

We have shown that equality of closed or coclosed $p$ forms for different
choices of $G,$ keeping $n,p,q$ constant, depends (iff) on equality
of the polynomial $H_{G}^{p}\left(z\right),$ where $H_{G}^{p}\left(z\right)$
is defined as:
\begin{align*}
 & H_{G}^{p}\left(z\right)=\\
 & \varPsi_{q_{1}}\left(z\right)^{\left(q_{2}-1\right)}\varPsi_{q_{2}}\left(z\right)^{\left(q_{1}-1\right)}\left[\sum_{c=-p}^{2k}\left(-1\right)^{c}z^{c}\sum_{a=\max\left(0,c\right)}^{\min\left(2k,c+p\right)}\left[C_{G}^{a,p-a+c,q}-C_{G}^{a,p-a+c,q_{1}}-C_{G}^{a,p-a+c,q_{2}}\right]\right.\\
 & \left.+\sum_{c=2}^{2k+p+2}\left(-1\right)^{c+1}z^{c}\sum_{a=\max\left(0,c-2-p\right)}^{\min\left(2k,c-2\right)}\left[C_{G}^{a,p+2+a-c,q}-C_{G}^{a,p+2+a-c,q_{1}}-C_{G}^{a,p+2+a-c,q_{2}}\right]\right]\\
 & +\varPsi_{q}\left(z\right)\varPsi_{q_{1}}\left(z\right)^{\left(q_{2}-1\right)}\left[\sum_{c=-p}^{2k}\left(-1\right)^{c}z^{c}\sum_{a=\max\left(0,c\right)}^{\min\left(2k,c+p\right)}C_{G}^{a,p-a+c,q_{1}}\right.\\
 & \left.+\sum_{c=2}^{2k+p+2}\left(-1\right)^{c+1}z^{c}\sum_{a=\max\left(0,c-2-p\right)}^{\min\left(2k,c-2\right)}\sum_{t=1,\left(t,q_{1}\right)\neq1}^{q}C_{G}^{a,p+2+a-c,q_{1}}\right]\\
 & +\varPsi_{q}\left(z\right)\varPsi_{q_{2}}\left(z\right)^{\left(q_{1}-1\right)}\left[\sum_{c=-p}^{2k}\left(-1\right)^{c}z^{c}\sum_{a=\max\left(0,c\right)}^{\min\left(2k,c+p\right)}C_{G}^{a,p-a+c,q_{2}}\right.\\
 & \left.+\sum_{c=2}^{2k+p+2}\left(-1\right)^{c+1}z^{c}\sum_{a=\max\left(0,c-2-p\right)}^{\min\left(2k,c-2\right)}C_{G}^{a,p+2+a-c,q_{2}}\right]
\end{align*}
which equals
\begin{align*}
 & H_{G}^{p}\left(z\right)=\\
 & \varPsi_{q_{1}}\left(z\right)^{\left(q_{2}-1\right)}\varPsi_{q_{2}}\left(z\right)^{\left(q_{1}-1\right)}\left[\sum_{c=-p}^{2k}\left(-1\right)^{c}z^{c}\sum_{a=\max\left(0,c\right)}^{\min\left(2k,c+p\right)}C_{G}^{a,p-a+c,q}\right.\\
 & \left.+\sum_{c=2}^{2k+p+2}\left(-1\right)^{c+1}z^{c}\sum_{a=\max\left(0,c-2-p\right)}^{\min\left(2k,c-2\right)}C_{G}^{a,p+2+a-c,q}\right]\\
 & +\left(\varPsi_{q}\left(z\right)\varPsi_{q_{1}}\left(z\right)^{\left(q_{2}-1\right)}-\varPsi_{q_{1}}\left(z\right)^{\left(q_{2}-1\right)}\varPsi_{q_{2}}\left(z\right)^{\left(q_{1}-1\right)}\right)\left[\sum_{c=-p}^{2k}\left(-1\right)^{c}z^{c}\sum_{a=\max\left(0,c\right)}^{\min\left(2k,c+p\right)}C_{G}^{a,p-a+c,q_{1}}\right.\\
 & \left.+\sum_{c=2}^{2k+p+2}\left(-1\right)^{c+1}z^{c}\sum_{a=\max\left(0,c-2-p\right)}^{\min\left(2k,c-2\right)}\sum_{t=1,\left(t,q_{1}\right)\neq1}^{q}C_{G}^{a,p+2+a-c,q_{1}}\right]\\
 & +\left(\varPsi_{q}\left(z\right)\varPsi_{q_{2}}\left(z\right)^{\left(q_{1}-1\right)}-\varPsi_{q_{1}}\left(z\right)^{\left(q_{2}-1\right)}\varPsi_{q_{2}}\left(z\right)^{\left(q_{1}-1\right)}\right)\left[\sum_{c=-p}^{2k}\left(-1\right)^{c}z^{c}\sum_{a=\max\left(0,c\right)}^{\min\left(2k,c+p\right)}C_{G}^{a,p-a+c,q_{2}}\right.\\
 & \left.+\sum_{c=2}^{2k+p+2}\left(-1\right)^{c+1}z^{c}\sum_{a=\max\left(0,c-2-p\right)}^{\min\left(2k,c-2\right)}C_{G}^{a,p+2+a-c,q_{2}}\right]
\end{align*}

Remark: We can ignore the components of the polynomials generated
by $\left(\begin{array}{c}
2k\\
2k-\alpha
\end{array}\right)\left(\begin{array}{c}
2n\\
2n-\beta
\end{array}\right),$ since these generate polynmials that depend only on $n,p,q$ and
are equal for distinct choices of $G.$\newpage{}

Letting $q=q_{1}^{2},$ where $q_{1}$is an odd proime. Then 
\begin{align*}
\tilde{C}_{G}^{\alpha,\beta,q_{1}^{2}}= & \sum_{\sigma,\bar{\sigma}}\sum_{t=1,\left(t,q\right)=1}^{q_{1}^{2}-1}\exp\left(\frac{2\pi it}{q}\sum\left(\bar{\sigma}+\sigma\right)\right)\\
 & =\sum_{\sigma,\bar{\sigma}}\left[\sum_{t=1}^{q_{1}^{2}}\exp\left(\frac{2\pi it}{q}\sum\left(\bar{\sigma}+\sigma\right)\right)-\sum_{t'=1}^{q_{1}}\exp\left(\frac{2\pi it'q_{1}}{q_{1}^{2}}\sum\left(\bar{\sigma}+\sigma\right)\right)\right]\\
 & =\sum_{\sigma,\bar{\sigma}}\left[\sum_{t=1}^{q}\exp\left(\frac{2\pi it}{q}\sum\left(\bar{\sigma}+\sigma\right)\right)\right]-\tilde{C}_{G}^{\alpha,\beta,q_{1}}\\
\tilde{C}_{G}^{\alpha,\beta,q_{1}}= & \sum_{\sigma,\bar{\sigma}}\sum_{t'=1,}^{q_{1}-1}\exp\left(\frac{2\pi it'q_{1}}{q_{1}^{2}}\sum\left(\bar{\sigma}+\sigma\right)\right)\\
 & =\sum_{\sigma,\bar{\sigma}}\sum_{t'=1}^{q_{1}}\exp\left(\frac{2\pi it'}{q_{1}}\sum\left(\bar{\sigma}+\sigma\right)\right)-\left(\begin{array}{c}
2k\\
2k-\alpha
\end{array}\right)\left(\begin{array}{c}
2n\\
2n-\beta
\end{array}\right)
\end{align*}
For $\ell\in\mathbb{Z},$ 
\[
\begin{cases}
\sum_{t=1}^{q_{1}^{2}}\exp\left(\frac{2\pi it}{q_{1}^{2}}\ell\right)=0 & \textrm{if }\left(\ell,q_{1}\right)=1,\\
\sum_{t=1}^{q_{1}^{2}}\exp\left(\frac{2\pi it}{q_{1}^{2}}\ell\right)=0 & \textrm{if }\ell\in q_{1}\mathbb{Z}\textrm{ and }\text{\ensuremath{\ell\notin q_{1}^{2}\mathbb{Z}}}\\
\sum_{t=1}^{q_{1}^{2}}\exp\left(\frac{2\pi it}{q_{1}^{2}}\ell\right)=q_{1}^{2} & \textrm{if }\ell\in q_{1}^{2}\mathbb{Z}
\end{cases}.
\]
The first line follows from the fact that the sum of all $q$ of the
$qth$ roots of unity is $0.$ The second line follow form the fact
that $\ell=q_{1}\ell'$ and still $\left(\ell',q_{1}\right)=1.$ Now
$\sum_{t=1}^{q_{1}^{2}}\exp\left(\frac{2\pi it}{q_{1}}\ell'\right)=0$
since it is just $q_{1}$ sums of the $q_{1}^{th}$ roots of unity.
The last line follow from the fact that $\exp\left(2\pi it\right)=1$
for all integers $t.$ We thus define
\begin{align*}
C_{G}^{\alpha,\beta,q_{1}}= & q_{1}\left(\textrm{number of pairs of sets \ensuremath{\sigma,\bar{\sigma}}such that \ensuremath{\sum_{\sigma,\bar{\sigma}}\left(\bar{\sigma}+\sigma\right)\in q_{1}\mathbb{Z}}}\right)\\
C_{G}^{\alpha,\beta,q_{1}^{2}}= & q_{1}^{2}\left(\textrm{number of pairs of sets \ensuremath{\sigma,\bar{\sigma}}such that \ensuremath{\sum_{\sigma,\bar{\sigma}}\left(\bar{\sigma}+\sigma\right)\in q_{1}^{2}\mathbb{Z}}}\right)\textrm{ }
\end{align*}
where $\bar{\sigma}$ runs over the $\left(\begin{array}{c}
2k\\
2k-\alpha
\end{array}\right)$ choices of subsets of order $2k-\alpha$ of the set 
\[
RPS_{q}:=RP_{q}-\left(\pm S\right)
\]
 and $\sigma$ runs over the $\left(\begin{array}{c}
2n\\
2n-\beta
\end{array}\right)$ choices of subsets of order $2n-\beta$ of the set 
\[
RPR_{q}:=RP_{q}-\left(\pm R\right).
\]
We have shown that equality of closed or coclosed $p$ forms for different
choices of $G,$ keeping $n,p,q$ constant, depends (iff) on equality
of the polynomial $H_{G}^{p}\left(z\right),$ where $H_{G}^{p}\left(z\right)$
is defined as:
\begin{align*}
H_{G}^{p}\left(z\right)= & \varPsi_{q_{1}}\left(z\right)^{q_{1}}\left[\sum_{c=-p}^{2k}\left(-1\right)^{c}z^{c}\sum_{a=\max\left(0,c\right)}^{\min\left(2k,c+p\right)}\left[C_{G}^{a,p-a+c,q_{1}^{2}}-C_{G}^{a,p-a+c,q_{1}}\right]\right.\\
 & \left.+\sum_{c=2}^{2k+p+2}\left(-1\right)^{c+1}z^{c}\sum_{a=\max\left(0,c-2-p\right)}^{\min\left(2k,c-2\right)}\left[C_{G}^{a,p+2+a-c,q_{1}^{2}}-C_{G}^{a,p+2+a-c,q_{1}}\right]\right]\\
 & +\varPsi_{q_{1}}\left(z^{q_{1}}\right)\left[\sum_{c=-p}^{2k}\left(-1\right)^{c}z^{c}\sum_{a=\max\left(0,c\right)}^{\min\left(2k,c+p\right)}C_{G}^{a,p-a+c,q_{1}}\right.\\
 & \left.+\sum_{c=2}^{2k+p+2}\left(-1\right)^{c+1}z^{c}\sum_{a=\max\left(0,c-2-p\right)}^{\min\left(2k,c-2\right)}C_{G}^{a,p+2+a-c,q_{1}}\right]
\end{align*}
which equals 
\begin{align*}
H_{G}^{p}\left(z\right)= & \varPsi_{q_{1}}\left(z\right)^{q_{1}}\left[\sum_{c=-p}^{2k}\left(-1\right)^{c}z^{c}\sum_{a=\max\left(0,c\right)}^{\min\left(2k,c+p\right)}C_{G}^{a,p-a+c,q_{1}^{2}}\right.\\
 & \left.+\sum_{c=2}^{2k+p+2}\left(-1\right)^{c+1}z^{c}\sum_{a=\max\left(0,c-2-p\right)}^{\min\left(2k,c-2\right)}C_{G}^{a,p+2+a-c,q_{1}^{2}}\right]\\
 & +\left(\varPsi_{q_{1}}\left(z^{q_{1}}\right)-\varPsi_{q_{1}}\left(z\right)^{q_{1}}\right)\left[\sum_{c=-p}^{2k}\left(-1\right)^{c}z^{c}\sum_{a=\max\left(0,c\right)}^{\min\left(2k,c+p\right)}C_{G}^{a,p-a+c,q_{1}}\right.\\
 & \left.+\sum_{c=2}^{2k+p+2}\left(-1\right)^{c+1}z^{c}\sum_{a=\max\left(0,c-2-p\right)}^{\min\left(2k,c-2\right)}C_{G}^{a,p+2+a-c,q_{1}}\right]
\end{align*}

\newpage Letting $q=q_{1}^{3},$ where $q_{1}$is an odd prime. Then
\begin{align*}
 & \tilde{C}_{G}^{\alpha,\beta,q_{1}^{3}}=\\
 & \sum_{\sigma,\bar{\sigma}}\sum_{t=1,\left(t,q_{1}\right)=1}^{q_{1}^{3}-1}\exp\left(\frac{2\pi it}{q}\sum\left(\bar{\sigma}+\sigma\right)\right)\\
 & =\sum_{\sigma,\bar{\sigma}}\left[\sum_{t=1}^{q_{1}^{3}}\exp\left(\frac{2\pi it}{q_{1}^{3}}\sum\left(\bar{\sigma}+\sigma\right)\right)-\sum_{t'=1,\left(t',q_{1}\right)\in\left\{ 1,q_{1}^{2}\right\} }^{q_{1}^{2}-1}\exp\left(\frac{2\pi it'q_{1}}{q_{1}^{3}}\sum\left(\bar{\sigma}+\sigma\right)\right)\right.\\
& \left. \qquad -\sum_{t'=1}^{q_{1}-1}\exp\left(\frac{2\pi it'q_{1}^{2}}{q_{1}^{3}}\sum\left(\bar{\sigma}+\sigma\right)\right)-1\right]\\
 & =\sum_{\sigma,\bar{\sigma}}\left[\sum_{t=1}^{q_{1}^{3}}\exp\left(\frac{2\pi it}{q_{1}^{3}}\sum\left(\bar{\sigma}+\sigma\right)\right)\right]-\tilde{C}_{G}^{\alpha,\beta,q_{1}^{2}}-\tilde{C}_{G}^{\alpha,\beta,q_{1}}-\left(\begin{array}{c}
2k\\
2k-\alpha
\end{array}\right)\left(\begin{array}{c}
2n\\
2n-\beta
\end{array}\right)\\
 & \tilde{C}_{G}^{\alpha,\beta,q_{1}^{2}}=\\
 & \sum_{\sigma,\bar{\sigma}}\sum_{t'=1,\left(t',q_{1}\right)\in\left\{ 1,q_{1}^{2}\right\} }^{q_{1}^{2}-1}\exp\left(\frac{2\pi it}{q}\sum\left(\bar{\sigma}+\sigma\right)\right)\\
 & =\sum_{\sigma,\bar{\sigma}}\left[\sum_{t=1}^{q_{1}^{2}}\exp\left(\frac{2\pi itq_{1}}{q_{1}^{3}}\sum\left(\bar{\sigma}+\sigma\right)\right)-\sum_{t'=1}^{q_{1}-1}\exp\left(\frac{2\pi it'q_{1}^{2}}{q_{1}^{3}}\sum\left(\bar{\sigma}+\sigma\right)\right)-1\right]\\
 & =\sum_{\sigma,\bar{\sigma}}\left[\sum_{t=1}^{q_{1}^{2}}\exp\left(\frac{2\pi it}{q_{1}^{2}}\sum\left(\bar{\sigma}+\sigma\right)\right)\right]-\tilde{C}_{G}^{\alpha,\beta,q_{1}}-\left(\begin{array}{c}
2k\\
2k-\alpha
\end{array}\right)\left(\begin{array}{c}
2n\\
2n-\beta
\end{array}\right)\\
 & \tilde{C}_{G}^{\alpha,\beta,q_{1}}=\\
 & \sum_{\sigma,\bar{\sigma}}\sum_{t'=1,}^{q_{1}-1}\exp\left(\frac{2\pi it'q_{1}^{2}}{q_{1}^{3}}\sum\left(\bar{\sigma}+\sigma\right)\right)\\
 & =\sum_{\sigma,\bar{\sigma}}\sum_{t'=1}^{q_{1}}\exp\left(\frac{2\pi it'}{q_{1}}\sum\left(\bar{\sigma}+\sigma\right)\right)-\left(\begin{array}{c}
2k\\
2k-\alpha
\end{array}\right)\left(\begin{array}{c}
2n\\
2n-\beta
\end{array}\right)
\end{align*}
For $\ell\in\mathbb{Z},$ 
\[
\begin{cases}
\sum_{t=1}^{q_{1}^{3}}\exp\left(\frac{2\pi it}{q_{1}^{3}}\ell\right)=0 & \textrm{if }\left(\ell,q_{1}\right)=1,\\
\sum_{t=1}^{q_{1}^{3}}\exp\left(\frac{2\pi it}{q_{1}^{3}}\ell\right)=0 & \textrm{if }\ell\in q_{1}\mathbb{Z}\textrm{ and }\text{\ensuremath{\ell\notin q_{1}^{2}\mathbb{Z}}}\\
\sum_{t=1}^{q_{1}^{3}}\exp\left(\frac{2\pi it}{q_{1}^{3}}\ell\right)=0 & \textrm{if }\ell\in q_{1}^{2}\mathbb{Z}\textrm{ and }\text{\ensuremath{\ell\notin q_{1}^{3}\mathbb{Z}}}\\
\sum_{t=1}^{q_{1}^{3}}\exp\left(\frac{2\pi it}{q_{1}^{3}}\ell\right)=q_{1}^{3} & \textrm{if }\ell\in q_{1}^{3}\mathbb{Z}
\end{cases}.
\]
The first line follows from the fact that the sum of all $q$ of the
$qth$ roots of unity is $0.$ The second line follows from the fact
that $\ell=q_{1}\ell'$ and still $\left(\ell',q_{1}\right)=1.$ Now
$\sum_{t=1}^{q_{1}^{3}}\exp\left(\frac{2\pi it}{q_{1}^{2}}\ell'\right)=0$
since it is just $q_{1}$ sums of the $q_{1}^{2th}$ roots of unity.
The third line follows from the fact that $\ell=q_{1}^{2}\ell'$ and
still $\left(\ell',q_{1}\right)=1.$ Now $\sum_{t=1}^{q_{1}^{3}}\exp\left(\frac{2\pi it}{q_{1}}\ell'\right)=0$
since it is just $q_{1}^{2}$ sums of the $q_{1}^{th}$ roots of unity.
The last line follow from the fact that $\exp\left(2\pi it\right)=1$
for all integers $t.$ We thus define 
\begin{align*}
C_{G}^{\alpha,\beta,q_{1}}= & q_{1}\left(\textrm{number of pairs of sets \ensuremath{\sigma,\bar{\sigma}}such that \ensuremath{\sum_{\sigma,\bar{\sigma}}\left(\bar{\sigma}+\sigma\right)\in q_{1}\mathbb{Z}}}\right)\\
C_{G}^{\alpha,\beta,q_{1}^{2}}= & q_{1}^{2}\left(\textrm{number of pairs of sets \ensuremath{\sigma,\bar{\sigma}}such that \ensuremath{\sum_{\sigma,\bar{\sigma}}\left(\bar{\sigma}+\sigma\right)\in q_{1}^{2}\mathbb{Z}}}\right)\textrm{ }\\
C_{G}^{\alpha,\beta,q_{1}^{3}}= & q_{1}^{3}\left(\textrm{number of pairs of sets \ensuremath{\sigma,\bar{\sigma}}such that \ensuremath{\sum_{\sigma,\bar{\sigma}}\left(\bar{\sigma}+\sigma\right)\in q_{1}^{3}\mathbb{Z}}}\right)\textrm{ }
\end{align*}
where $\bar{\sigma}$ runs over the $\left(\begin{array}{c}
2k\\
2k-\alpha
\end{array}\right)$ choices of subsets of order $2k-\alpha$ of the set 
\[
RPS_{q}:=RP_{q}-\left(\pm S\right)
\]
 and $\sigma$ runs over the $\left(\begin{array}{c}
2n\\
2n-\beta
\end{array}\right)$ choices of subsets of order $2n-\beta$ of the set 
\[
RPR_{q}:=RP_{q}-\left(\pm R\right).
\]
We have shown that equality of closed or coclosed $p$ forms for different
choices of $G,$ keeping $n,p,q$ constant, depends (iff) on equality
of the polynomial $H_{G}^{p}\left(z\right),$ where $H_{G}^{p}\left(z\right)$
is defined as:
\begin{align*}
H_{G}^{p}\left(z\right)= & \varPsi_{q_{1}}\left(z^{q_{1}}\right)^{q_{1}}\varPsi_{q_{1}}\left(z\right)^{q_{1}^{2}}\left[\sum_{c=-p}^{2k}\left(-1\right)^{c}z^{c}\sum_{a=\max\left(0,c\right)}^{\min\left(2k,c+p\right)}\left(C_{G}^{a,p-a+c,q_{1}^{3}}-C_{G}^{a,p-a+c,q_{1}^{2}}-C_{G}^{a,p-a+c,q_{1}}\right)\right.\\
 & \left.+\sum_{c=2}^{2k+p+2}\left(-1\right)^{c+1}z^{c}\sum_{a=\max\left(0,c-2-p\right)}^{\min\left(2k,c-2\right)}\left(C_{G}^{a,p+2+a-c,q_{1}^{3}}-C_{G}^{a,p+2+a-c,q_{1}^{2}}-C_{G}^{a,p+2+a-c,q_{1}}\right)\right]\\
 & +\varPsi_{q_{1}}\left(z^{q_{1}^{2}}\right)\varPsi_{q_{1}}\left(z\right)^{q_{1}^{2}}\left[\sum_{c=-p}^{2k}\left(-1\right)^{c}z^{c}\sum_{a=\max\left(0,c\right)}^{\min\left(2k,c+p\right)}\left(C_{G}^{a,p-a+c,q_{1}^{2}}-C_{G}^{a,p-a+c,q_{1}}\right)\right.\\
 & \left.+\sum_{c=2}^{2k+p+2}\left(-1\right)^{c+1}z^{c}\sum_{a=\max\left(0,c-2-p\right)}^{\min\left(2k,c-2\right)}\left(C_{G}^{a,p+2+a-c,q_{1}^{2}}-C_{G}^{a,p+2+a-c,q_{1}}\right)\right]\\
 & +\varPsi_{q_{1}}\left(z^{q_{1}^{2}}\right)\varPsi_{q_{1}}\left(z^{q_{1}}\right)^{q_{1}}\left[\sum_{c=-p}^{2k}\left(-1\right)^{c}z^{c}\sum_{a=\max\left(0,c\right)}^{\min\left(2k,c+p\right)}C_{G}^{a,p-a+c,q_{1}}\right.\\
 & \left.+\sum_{c=2}^{2k+p+2}\left(-1\right)^{c+1}z^{c}\sum_{a=\max\left(0,c-2-p\right)}^{\min\left(2k,c-2\right)}C_{G}^{a,p+2+a-c,q_{1}}\right]
\end{align*}
which equals 
\begin{align*}
H_{G}^{p}\left(z\right)= & \varPsi_{q_{1}}\left(z^{q_{1}}\right)^{q_{1}}\varPsi_{q_{1}}\left(z\right)^{q_{1}^{2}}\left[\sum_{c=-p}^{2k}\left(-1\right)^{c}z^{c}\sum_{a=\max\left(0,c\right)}^{\min\left(2k,c+p\right)}C_{G}^{a,p-a+c,q_{1}^{3}}\right.\\
 & \left.+\sum_{c=2}^{2k+p+2}\left(-1\right)^{c+1}z^{c}\sum_{a=\max\left(0,c-2-p\right)}^{\min\left(2k,c-2\right)}C_{G}^{a,p+2+a-c,q_{1}^{3}}\right]\\
 & +\left(\varPsi_{q_{1}}\left(z^{q_{1}^{2}}\right)\varPsi_{q_{1}}\left(z\right)^{q_{1}^{2}}-\varPsi_{q_{1}}\left(z^{q_{1}}\right)^{q_{1}}\varPsi_{q_{1}}\left(z\right)^{q_{1}^{2}}\right)\left[\sum_{c=-p}^{2k}\left(-1\right)^{c}z^{c}\sum_{a=\max\left(0,c\right)}^{\min\left(2k,c+p\right)}C_{G}^{a,p-a+c,q_{1}^{2}}\right.\\
 & \left.+\sum_{c=2}^{2k+p+2}\left(-1\right)^{c+1}z^{c}\sum_{a=\max\left(0,c-2-p\right)}^{\min\left(2k,c-2\right)}C_{G}^{a,p+2+a-c,q_{1}^{2}}\right]\\
 & +\left(\varPsi_{q_{1}}\left(z^{q_{1}^{2}}\right)\varPsi_{q_{1}}\left(z^{q_{1}}\right)^{q_{1}}-\varPsi_{q_{1}}\left(z^{q_{1}^{2}}\right)\varPsi_{q_{1}}\left(z\right)^{q_{1}^{2}}-\varPsi_{q_{1}}\left(z^{q_{1}}\right)^{q_{1}}\varPsi_{q_{1}}\left(z\right)^{q_{1}^{2}}\right)\times\\
& \left[\sum_{c=-p}^{2k}\left(-1\right)^{c}z^{c}\sum_{a=\max\left(0,c\right)}^{\min\left(2k,c+p\right)}C_{G}^{a,p-a+c,q_{1}}+\sum_{c=2}^{2k+p+2}\left(-1\right)^{c+1}z^{c}\sum_{a=\max\left(0,c-2-p\right)}^{\min\left(2k,c-2\right)}C_{G}^{a,p+2+a-c,q_{1}}\right]
\end{align*}

\begin{rem}
 Assuming here that $q=q_{1}q_{2}$, and we're summing
as in formula 4.10 in \cite{Ik1}. 
\[
\sum_{t=1}^{q}\zeta^{tx}
\]
where $\zeta=e^{2\pi i/q}$. In our case the sum is only over values
where the GCD $=1$. First, 
\begin{eqnarray}
\sum_{t=1,GCD(t,q)=1}^{q_{1}q_{2}}e^{(2\pi i/q)tx} & = & \mbox{ if \ensuremath{x=kq_{1}q_{2}}}\\
\sum_{t=1,GCD(t,q)=1}^{q_{1}q_{2}}e^{(2\pi ik)t} & = & \sum_{t=1,GCD(t,q)=1}^{q_{1}q_{2}}1^{t}=\phi(q)
\end{eqnarray}
So, for that sum, we count how many ways to add choices to 0 mod $q$,
and multiply by $\phi(q)$. Second, 
\begin{eqnarray}
\sum_{t=1,GCD(t,q_{1})\neq1}^{q_{1}q_{2}}e^{(2\pi i/q)tx} & = & \sum_{i=1}^{q_{2}}e^{(2\pi i/q_{2})x}=\mbox{since \ensuremath{t} is a multiple of \ensuremath{q_{1}}}\\
q_{2} & = & \mbox{if \ensuremath{x=kq_{2}} since \ensuremath{q_{2}} is prime and we're summing for 1 to \ensuremath{q2}}
\end{eqnarray}
So for this one, we count how many ways to add to 0 mod $q_{2}$,
and multiply by $q2$. Third, 
\begin{eqnarray}
\sum_{t=1,GCD(t,q_{2})\neq1}^{q_{1}q_{2}}e^{(2\pi i/q)tx} & = & \sum_{i=1}^{q_{1}}e^{(2\pi i/q_{1})x}=\mbox{since \ensuremath{t} is a multiple of \ensuremath{q_{2}}}\\
q_{1} & = & \mbox{if \ensuremath{x=kq_{1}} since \ensuremath{q_{1}} is prime and we're summing for 1 to \ensuremath{q_{1}}}
\end{eqnarray}
So we count how many ways to add to 0 mod $q_{1}$ and multiply by
$q_{1}$.
\end{rem}

\section{COMPUTATIONAL RESULTS}
The calculations involved in looking for lens spaces with unusual
isospectralities become quite a bit more involved than we previously thought. We modified our
original Mathematica code, and were able to carry out for some values
of $q$, but Mathematic is too slow to allow us to do any real calculations for reasonable values of $q$. Therefore,
we wrote compilable Swift code and used that to extend the calculations.
In the end, we obtained examples of everything we (mistakenly) thought
we had before, including examples of lens spaces isospectral on forms
but not functions.

The basis of all the calculations is considering the set of numbers
relatively prime to $q$, and splitting it into two sets, one of size
$2k$. As $k$ increases, the calculations increase in difficulty
and time quickly - in our original work we only considered $k=2,3$
and $q<100$. Our current work allowed us to increase $k$ to 7 for
$q<100$, and to do calculations for prime $q$ with $k\leq3$ up
to $q=200$. For composite $q$ we are still limited to $q<100$.

We list below examples of spaces which are isospectral on forms but
not functions, together with a few examples isospectral for functions
and some sporadic degree forms. We include the defining $2k$ tuples of numbers
relatively prime to $q$. Note that our most interesting examples
are when $q$ is prime, and hence could have been found with our original
work if our code had been quicker, and computers faster. We also note
that no examples of isospectrality on forms but not functions was
found for composite $q$, but that might simply be because we could
not get to high enough values of $q$ or $k$. We did find all sorts
of weird isospectrality in the composite case, there are many examples which look like the last line in the table.  
\footnotesize{
\begin{center}
\begin{tabular}{|c|c|c|c|c|}
\hline 
$q$ & $k$ & isosp for forms of degree & first set & second set\tabularnewline
\hline 
\hline 
$59$ & $5$ & $2$ & $[16,25,4,9,60,57,36,52,45,1]$ & $[19,22,25,55,39,60,6,36,1,42]$\tabularnewline
\hline 
$61$ & $5$ & $2$ & $[16,58,28,31,7,52,48,43,1,11]$ & $[56,17,19,32,58,40,27,3,1,42]$\tabularnewline
\hline 
$67$ & $5$ & $2$ & $[18,49,40,38,27,15,52,29,66,1]$ & $[12,17,55,60,40,27,7,50,66,1]$\tabularnewline
\hline 
$65$ & $3$ & $0,1,12$ & $[31,34,64,9,1,56]$ & $[36,41,29,24,64,1]$\tabularnewline
\hline 
\end{tabular}
\par\end{center}
}
\begin{rem}
We note that this work could be extended for $q$ a product of more
than 2 primes in an obvious, if messy, way. The end result of all
these calculation is that equality of the $C^{\alpha,\beta}$ from
\cite{GM} is replaced by equality of the $K^{p,q},K^{p,q1}$ and
$K^{p,q2}$ (see (\ref{Kdefns})). The subset sum calculations from
\cite{GM} are modified as follows.
\end{rem}

\section{Swift Code}

We include here swift code for two cases.  The first is for prime $q$.  The second is for $q$ which are the product of two primes.   Using the swift package BIGINT allows for computations with unlimited size integers.

\subsection{$q$ prime}
\tiny{
\begin{verbatim}


import Cocoa
import Foundation


var schoices: Set<Set<Int>> = [[]]
var schoicesarray: Array<Set<Int>> = []     // So I can find out what k-tuples match
var subsetlist: Array<Array<Int>>
var q: Int = 101
var qq: Int = 101  // for multiplying a Int value
var q0: Int = 50

var k: Int = 5
var n: Int = q0 - k
var relprime: Array<Int>
var i = 2


relprime = []

func max(_ num1: Int,_ num2: Int) -> Int {
    
    if num1 >= num2 {
        return num1
    } else {
        return num2
    }
}

func min(_ num1: Int,_ num2: Int) -> Int {
    
    if num1 <= num2 {
        return num1
    } else {
        return num2
    }
}

func GCD(_ first: Int, _ second: Int)->Int {
    var f = first
    var s = second
    
    while f>0 && s>0 {
        if f < s {
            s = s-f
        } else {
            f = f-s
        }
    }
    if f == 0 {
        return s
    } else {
        return f
    }
}

struct polynomial {  
    
    var coefs: [BigInt] = []
    var degree: Int {
        get {
            return coefs.count - 1
        }
    }
    
    func multiplypolys(_ poly1: polynomial) -> polynomial {
        
        var newpoly = polynomial()
        var degreetracker: Int = 0
        
        for multindex in 0...poly1.degree + self.degree {
            
            newpoly.coefs.append(0)
            
            while degreetracker <= multindex && degreetracker <= poly1.degree {  //multiply all possible combinations add to correct degree
                
                if multindex - degreetracker <= self.degree {
                    
                    newpoly.coefs[multindex] += (poly1.coefs[degreetracker]*self.coefs[multindex - degreetracker])
                }
                
                degreetracker += 1
                
            }
            
            degreetracker = 0
        }
        
        return newpoly
    }
    
    func addpolys(_ poly1: polynomial) -> polynomial{
        
        var newpoly = polynomial()
        
        for addindex in 0...max(poly1.degree,self.degree) {
            
            newpoly.coefs.append(0)
            
            if (addindex <= min(poly1.degree,self.degree)) {  // if can add both, do it, otherwise pick correct one
                
                newpoly.coefs[addindex] = (poly1.coefs[addindex] + self.coefs[addindex])
                
            } else if addindex <= poly1.degree && addindex > self.degree {
                
                newpoly.coefs[addindex] = poly1.coefs[addindex]
                
            } else if addindex > poly1.degree && addindex <= self.degree {
                
                newpoly.coefs[addindex] = self.coefs[addindex]
                
            }
        }
        return newpoly
    }
}

struct polynomialmodq {
    
    var coefs: [BigInt] = []
    var degree: Int {
        get {
            return coefs.count - 1
        }
    }
    
    func multiplypolys(_ poly1: polynomialmodq) -> polynomialmodq {  // returns the product with powers mod q
        
        var newpoly = polynomialmodq()
        var degreetracker: Int = 0
        var multiplier: Int = 1     // to get multiples of numbers to combine powers mod q
        
        for multindex in 0...poly1.degree + self.degree {
            
            newpoly.coefs.append(0)
            
            while degreetracker <= multindex && degreetracker <= poly1.degree {  //multiply all possible combinations add to correct degree
                
                if multindex - degreetracker <= self.degree {
                    
                    newpoly.coefs[multindex] += (poly1.coefs[degreetracker]*self.coefs[multindex - degreetracker])
                    
                }
                
                degreetracker += 1
                
            }
            
            degreetracker = 0
        }
        
        // knock powers >= q down and add them in
        
        if newpoly.degree >= q {
            
            for powerindex in 0...q - 1 {
                
                while multiplier*q + powerindex <= newpoly.degree {
                    
                    newpoly.coefs[powerindex % q] += newpoly.coefs[multiplier*q + powerindex]
                    multiplier += 1
                    
                }
                multiplier = 1
            }
            
            for powerindex in q...newpoly.degree {      // kill everything from x^q on
                
                newpoly.coefs.removeLast()
                
            }
            
        }
        
        return newpoly
    }
    
    func addpolys(_ poly1: polynomialmodq) -> polynomialmodq {        // returns the sum with powers mod q
        
        var newpoly = polynomialmodq()
        var multiplier: Int = 1     // to get multiples of numbers to combine powers mod q
        
        for addindex in 0...max(poly1.degree,self.degree) {
            
            newpoly.coefs.append(0)
            
            if (addindex <= min(poly1.degree,self.degree)) {  // if can add both, do it, otherwise pick correct one
                
                newpoly.coefs[addindex] = (poly1.coefs[addindex] + self.coefs[addindex])
                
            } else if addindex <= poly1.degree && addindex > self.degree {
                
                newpoly.coefs[addindex] = poly1.coefs[addindex]
                
            } else if addindex > poly1.degree && addindex <= self.degree {
                
                newpoly.coefs[addindex] = self.coefs[addindex]
                
            }
        }
        
        
        // knock powers >= q down and add them in
        
        if newpoly.degree >= q {
            
            for powerindex in 0...q - 1 {
                
                while multiplier*q + powerindex <= newpoly.degree {
                    
                    newpoly.coefs[powerindex % q] += newpoly.coefs[multiplier*q + powerindex]
                    multiplier += 1
                    
                }
                multiplier = 1
            }
            
            for powerindex in q...newpoly.degree {      // kill everything from x^q on
                
                newpoly.coefs.removeLast()
                
            }
            
        }
        return newpoly
    }
}

var poly1 = polynomialmodq(coefs: [1])      // the poly = 1
var poly0 = polynomialmodq(coefs: [0])      // the poly = 0

struct doublepolynomial { // polynomials with polynomial coefficients, powers mod q
    
    var coefs: [polynomialmodq] = []
    var degree: Int {
        get {
            return coefs.count - 1
        }
    }
    
    func multiplydoublepolys(_ poly1: doublepolynomial) -> doublepolynomial {
        
        var newpoly = doublepolynomial()
        var degreetracker: Int = 0
        let zeropoly = polynomialmodq(coefs: [0])
        
        for multindex in 0...poly1.degree + self.degree {
            
            newpoly.coefs.append(zeropoly)
            
            while degreetracker <= multindex && degreetracker <= poly1.degree {
                
                if multindex - degreetracker <= self.degree {
                    
                    newpoly.coefs[multindex] = newpoly.coefs[multindex].addpolys(poly1.coefs[degreetracker].multiplypolys(self.coefs[multindex - degreetracker]))       // add all possible products for each degree
                    
                }
                
                degreetracker += 1
                
            }
            
            degreetracker = 0
        }
        
        return newpoly
    }
    
    func multiplybybinomial(_ xpower: Int) -> doublepolynomial {  //  FIX, need to foil for y coefs
        
        // multiply this double by 1+x^n*y, for speed
        
        var arraytopadwith: [BigInt] = []   // for padding front of poly to multiply by x^n
        var bufferpoly = polynomialmodq()   //for holding a copy of the coef polynomial which can be padded
        var newdoublepoly = doublepolynomial()
        
        for powerindex in 1...xpower {  // set up array to multiply by x^powerindex
            
            arraytopadwith.append(0)
            
        }
        
        newdoublepoly.coefs.append(self.coefs[0])      //first coef doesn't change
        
        if self.degree != 0 {for polyindex in 1...self.degree {
            
            bufferpoly = self.coefs[polyindex - 1]      //new coef is current one plus x^n times previous one
            bufferpoly.coefs.insert(contentsOf: arraytopadwith, at: 0)
            newdoublepoly.coefs.append(bufferpoly.addpolys(self.coefs[polyindex]))
            
            }}
        
        // and finally add the last coef
        
        bufferpoly = self.coefs[self.degree]
        bufferpoly.coefs.insert(contentsOf: arraytopadwith, at: 0)
        newdoublepoly.coefs.append(bufferpoly.addpolys(poly0))      // knock x powers down mod q
        
        return newdoublepoly
        
    }
}



struct etapart {
    
    var powers: [Int]
    var coefs: [BigInt]
    
    func addetaparts(_ toaddto: etapart) -> etapart {
        
        var newetaparts = etapart(powers: [], coefs: [])
        var indexofmatching: Int?
        
        for firstpowerindex in 0...self.powers.count - 1 {
            
            newetaparts.powers.append(self.powers[firstpowerindex])     // add first power to powers and first coef to coefs
            newetaparts.coefs.append(self.coefs[firstpowerindex])
            
        }
        
        
        for secondpowerindex in 0...toaddto.powers.count - 1 {
            
            if newetaparts.powers.contains(toaddto.powers[secondpowerindex]) {  // if second power already there
                
                indexofmatching = newetaparts.powers.index(of: toaddto.powers[secondpowerindex])
                
                newetaparts.coefs[indexofmatching!] += toaddto.coefs[secondpowerindex] // add coef to existing one
                
            } else {
                
                newetaparts.powers.append(toaddto.powers[secondpowerindex])     // otherwise just append it with coef
                newetaparts.coefs.append(toaddto.coefs[secondpowerindex])
                
            }
            
        }
        
        // sort powers and coefs based on powers
        
        let combined = zip(newetaparts.powers, newetaparts.coefs).sorted {$0.0 < $1.0}
        
        newetaparts.powers = combined.map {$0.0}
        newetaparts.coefs = combined.map {$0.1}
        
        return newetaparts
        
    }
}

while i <= q/2 {
    if GCD(i,q) == 1 {relprime.append(i)}
    i += 1
}

func subsets(_ source: [Int], takenBy : Int) -> [[Int]] {
    if(source.count == takenBy) {
        return [source]
    }
    
    if(source.isEmpty) {
        return []
    }
    
    if(takenBy == 0) {
        return []
    }
    
    if(takenBy == 1) {
        return source.map { [$0] }
    }
    
    var result : [[Int]] = []
    
    let rest = Array(source.suffix(from: 1))
    let sub_combos = subsets(rest, takenBy: takenBy - 1)
    result += sub_combos.map { [source[0]] + $0 }
    
    result += subsets(rest, takenBy: takenBy)
    
    return result
}

subsetlist = subsets(relprime,takenBy:k-1)
subsetlist.count


var nextschoice: Set<Int> = [1]    //next array to add to schoices
var subsetlistindex: Int = 0
var subsetindex: Int = 0

while subsetlistindex < subsetlist.count {
    while subsetindex < subsetlist[subsetlistindex].count {
        nextschoice.insert(subsetlist[subsetlistindex][subsetindex])
        nextschoice.insert(q - subsetlist[subsetlistindex][subsetindex])
        subsetindex += 1
    }
    nextschoice.insert(q - 1)
    schoices.insert(nextschoice)
    nextschoice = [1]
    subsetindex = 0
    subsetlistindex += 1
}
print("\(schoices.count)\n")

var mult: Int = 2
var currentschoice: Set<Int>
var currentschoicedup: Set<Int>
var donemultiplying: Bool = false

func multsetby(_ settomult: Set<Int>,multby: Int) -> Set<Int> {
    var setelement: Int = 0
    var localsettomult: Set<Int> = []
    
    for setelement in settomult {
        localsettomult.insert(multby*setelement % q)
    }
    
    return localsettomult
}



for currentschoice in schoices {
    
    if schoices.contains(currentschoice) {
        
        while mult < q/2 && !donemultiplying {
            
            currentschoicedup = multsetby(currentschoice, multby: mult)
            if schoices.contains(currentschoicedup) {
                schoices.remove(currentschoicedup)
                donemultiplying = true
            }
            mult += 1
        }
        mult = 2
        donemultiplying = false
    }
}


print("\(schoices.count)")

var complete: Set<Int> = [1,q-1]
var s: Set<Int>
var r: Set<Int>
var S = Array<Array<BigInt>>()
var R = Array<Array<BigInt>>()
S = Array(repeating: Array(repeating: 0, count: q), count: 2*k + 1)
R = Array(repeating: Array(repeating: 0, count: 2*n + 1), count: q)
var choices: Array<Array<Int>>
var pick:Int
var Gs = doublepolynomial(coefs: [poly1])
var Gr = doublepolynomial(coefs: [poly1])
var p = polynomialmodq()    //for picking off the coeffiecient polynomials from Gs and Gr

var sizes: Int = 2*k
var sizer: Int = 2*n

for relprimeint in relprime {
    complete.insert(relprimeint)
    complete.insert(q-relprimeint)
}

func totalset(_ settoadd: Array<Int>) -> Int {
    var total: Int = 0
    
    for number in settoadd {
        total += number
    }
    
    return total
}

func countoccurrences(_ inarray: [Int],_ whattofind: Int) -> Int {
    var currentcount: Int = 0
    
    for idx in inarray {
        if idx == whattofind {currentcount += 1}
    }
    
    return currentcount
}

func powint(_ base: BigInt,power: Int) -> BigInt {
    
    var answer: BigInt = 1
    
    if power != 0 {
        for _ in 1...power {
            
            answer *= base
            
        }
    }
    
    return answer
}

func compareetarows(_ etatocheck: Array<Array<etapart>>, row1: Int, row2: Int) -> [Bool] {  // check two rows of eta for equal values, return true if equal for given p (column)
    
    var part1: etapart
    var part2: etapart
    var comparelist: [Bool] = []
    
    for pindex in 0...n {
        
        part1 = etatocheck[row1][pindex]
        part2 = etatocheck[row2][pindex]
        
        if part1.coefs == part2.coefs && part1.powers == part2.powers {comparelist.append(true)} else {comparelist.append(false)}
        
    }
    
    return comparelist
}



var C = Array<Array<BigInt>>()
C = Array(repeating: Array(repeating: 0, count: sizer + 1), count: sizes + 1)
var rowtimescolumntotal: BigInt = 0
var sindex: Int = 0
var eta = Array(repeating: Array(repeating: etapart(powers: [0], coefs: [0]), count: n + 1), count: schoices.count)  //this is the identity for the etapart add
var newetapart = etapart(powers: [], coefs: [])


for s in schoices {  // loop through valid k-tuples
    
    r = complete.subtracting(s)
    
    for sint in s {
        
        Gs = Gs.multiplybybinomial(sint)
        
    }
    
    for rint in r {
        
        Gr = Gr.multiplybybinomial(rint)

    }
    
    for pick in 0...sizes {
        
        if Gs.coefs.count > pick {p = Gs.coefs[pick]} else {p = poly0}
        for sidx in 1...q {
            if p.coefs.count >= sidx {
                S[pick][sidx - 1] = p.coefs[sidx - 1]
            }
        }
    }
    
    for pick in 0...sizer {
        if Gr.coefs.count > pick {p = Gr.coefs[pick]} else {p = poly0}
        R[0][pick] = p.coefs[0]     //first R is the coef of x^0
        for ridx in 1...(q - 1) {   //the rest load backwards
            if p.coefs.count > ridx {
                R[q - ridx][pick] = p.coefs[ridx % q]
            }
        }
    }
    
    
    
    
    Gs.coefs = [poly1]
    Gr.coefs = [poly1]
    
    for rowinfirst in 0...sizes {
        for columninsecond in 0...sizer {
            
            rowtimescolumntotal = 0
            
            for j in 0...q-1 {
                
                rowtimescolumntotal += S[rowinfirst][j] * R[j][columninsecond]
            }
            
            C[rowinfirst][columninsecond] =  qq * rowtimescolumntotal
            
        }
    }
    

    
    for p in 0...n {
        for a in 0...2*k {
            for t in 0...p {
                
                newetapart = etapart(powers: [a - t,a + t + 2], coefs: [powint(-1,power: t+a) * C[a][p - t],-powint(-1,power: t+a) * C[a][p - t]])
                
                eta[sindex][p] = eta[sindex][p].addetaparts(newetapart)
                
            }
        }
    }
    
    schoicesarray.append(s)     // build array with k tuples in same order for later
    sindex += 1
    print(sindex)
}

var rowcompare: [Bool]
var matchlist: [[Int]] = []
var inequalrun: Bool = false        // set to true if I'm in a run of equal etaparts
var howmanyruns: Int = 0
var runs: Array<Int> = []
var runindex: Int = 0
var nonzerocount: Int = 0

for i in 0...schoices.count - 2 {
    
    for j in i + 1...schoices.count - 1 {
        
        rowcompare = compareetarows(eta, row1: i, row2: j)
        
        for kk in 0...rowcompare.count - 1 {
            
            if !inequalrun && rowcompare[kk] {
                
                matchlist.append([kk,0])
                howmanyruns += 1
                inequalrun = true
                
            } else if inequalrun && !rowcompare[kk] {
                
                matchlist[howmanyruns - 1][1] = kk - 1
                inequalrun = false
                
            }
        }
        
        while runindex < matchlist.count {
            
            runs.append(matchlist[runindex][1] - matchlist[runindex][0])
            if runs[runindex] > 0 {nonzerocount += 1}
            runindex += 1
            
            
        }
        
        if nonzerocount > 1 || (nonzerocount == 1 && matchlist[0][0] > 0) {print(matchlist, schoicesarray[i],schoicesarray[j])}
        
        howmanyruns = 0
        matchlist = []
        inequalrun = false
        runs = []
        runindex = 0
        nonzerocount = 0
    }
}
\end{verbatim}
\subsection{$q$ product of two primes}
\begin{verbatim}
//
//  main.swift
//  lens spaces product two primes
//
//  Created by Jeffrey Mcgowan on 10/21/16.




import Cocoa
import Foundation


var schoices: Set<Set<Int>> = [[]]
var schoicesarray: Array<Set<Int>> = []     // So I can find out what k-tuples match
var subsetlist: Array<Array<Int>>
var q: Int = 62
var qq: Int = 62  // for multiplying a Int value
var q1: Int = 2
var q2: Int = 31
var q0: Int = 15
var EulerPhiq = 30


var k: Int = 5
var n: Int = q0 - k
var relprime: Array<Int>
var i = 2
var prime = [29,31,37,41,43,47]

func max(_ num1: Int,_ num2: Int) -> Int {
    
    if num1 >= num2 {
        return num1
    } else {
        return num2
    }
}

func min(_ num1: Int,_ num2: Int) -> Int {
    
    if num1 <= num2 {
        return num1
    } else {
        return num2
    }
}

func GCD(_ first: Int, _ second: Int)->Int {
    var f = first
    var s = second
    
    while f>0 && s>0 {
        if f < s {
            s = s-f
        } else {
            f = f-s
        }
    }
    if f == 0 {
        return s
    } else {
        return f
    }
}

struct polynomial {
    
    var coefs: [Int] = []
    var degree: Int {
        get {
            return coefs.count - 1
        }
    }
    
    func multiplypolys(_ poly1: polynomial) -> polynomial {
        
        var newpoly = polynomial()
        var degreetracker: Int = 0
        
        for multindex in 0...poly1.degree + self.degree {
            
            newpoly.coefs.append(0)
            
            while degreetracker <= multindex && degreetracker <= poly1.degree {  //multiply all possible combinations add to correct degree
                
                if multindex - degreetracker <= self.degree {
                    
                    newpoly.coefs[multindex] += (poly1.coefs[degreetracker]*self.coefs[multindex - degreetracker])
                }
                
                degreetracker += 1
                
            }
            
            degreetracker = 0
        }
        
        return newpoly
    }
    
    func addpolys(_ poly1: polynomial) -> polynomial{
        
        var newpoly = polynomial()
        
        for addindex in 0...max(poly1.degree,self.degree) {
            
            newpoly.coefs.append(0)
            
            if (addindex <= min(poly1.degree,self.degree)) {  // if can add both, do it, otherwise pick correct one
                
                newpoly.coefs[addindex] = (poly1.coefs[addindex] + self.coefs[addindex])
                
            } else if addindex <= poly1.degree && addindex > self.degree {
                
                newpoly.coefs[addindex] = poly1.coefs[addindex]
                
            } else if addindex > poly1.degree && addindex <= self.degree {
                
                newpoly.coefs[addindex] = self.coefs[addindex]
                
            }
        }
        return newpoly
    }
}

var Cyclotomicq = polynomial(coefs: [1,-1,1,-1,1,-1,1,-1,1,-1,1,-1,1,-1,1,-1,1,-1,1,-1,1,-1,1,-1,1,-1,1,-1,1,-1,1])
var Cyclotomicq1 = polynomial(coefs: [1,1])
var Cyclotomicq2 = polynomial(coefs: [1,1,1,1,1,1,1,1,1,1,1,1,1,1,1,1,1,1,1,1,1,1,1,1,1,1,1,1,1,1,1])
var Cyclotomicq1q2 = (Cyclotomicq1.multiplypolys(Cyclotomicq2))
var Cyclotomicqq2 = (Cyclotomicq.multiplypolys(Cyclotomicq2))
var Cyclotomicq1q = (Cyclotomicq1.multiplypolys(Cyclotomicq))


struct polynomialmodq {
    
    var coefs: [Int] = []
    var degree: Int {
        get {
            return coefs.count - 1
        }
    }
    
    func multiplypolys(_ poly1: polynomialmodq) -> polynomialmodq {  // returns the product with powers mod q
        
        var newpoly = polynomialmodq()
        var degreetracker: Int = 0
        var multiplier: Int = 1     // to get multiples of numbers to combine powers mod q
        
        for multindex in 0...poly1.degree + self.degree {
            
            newpoly.coefs.append(0)
            
            while degreetracker <= multindex && degreetracker <= poly1.degree {  //multiply all possible combinations add to correct degree
                
                if multindex - degreetracker <= self.degree {
                    
                    newpoly.coefs[multindex] += (poly1.coefs[degreetracker]*self.coefs[multindex - degreetracker])
                    
                }
                
                degreetracker += 1
                
            }
            
            degreetracker = 0
        }
        
        // knock powers >= q down and add them in
        
        if newpoly.degree >= q {
            
            for powerindex in 0...q - 1 {
                
                while multiplier*q + powerindex <= newpoly.degree {
                    
                    newpoly.coefs[powerindex % q] += newpoly.coefs[multiplier*q + powerindex]
                    multiplier += 1
                    
                }
                multiplier = 1
            }
            
            for powerindex in q...newpoly.degree {      // kill everything from x^q on
                
                newpoly.coefs.removeLast()
                
            }
            
        }
        
        return newpoly
    }
    
    func addpolys(_ poly1: polynomialmodq) -> polynomialmodq {        // returns the sum with powers mod q
        
        var newpoly = polynomialmodq()
        var multiplier: Int = 1     // to get multiples of numbers to combine powers mod q
        
        for addindex in 0...max(poly1.degree,self.degree) {
            
            newpoly.coefs.append(0)
            
            if (addindex <= min(poly1.degree,self.degree)) {  // if can add both, do it, otherwise pick correct one
                
                newpoly.coefs[addindex] = (poly1.coefs[addindex] + self.coefs[addindex])
                
            } else if addindex <= poly1.degree && addindex > self.degree {
                
                newpoly.coefs[addindex] = poly1.coefs[addindex]
                
            } else if addindex > poly1.degree && addindex <= self.degree {
                
                newpoly.coefs[addindex] = self.coefs[addindex]
                
            }
        }
        
        
        // knock powers >= q down and add them in
        
        if newpoly.degree >= q {
            
            for powerindex in 0...q - 1 {
                
                while multiplier*q + powerindex <= newpoly.degree {
                    
                    newpoly.coefs[powerindex % q] += newpoly.coefs[multiplier*q + powerindex]
                    multiplier += 1
                    
                }
                multiplier = 1
            }
            
            for powerindex in q...newpoly.degree {      // kill everything from x^q on
                
                newpoly.coefs.removeLast()
                
            }
            
        }
        return newpoly
    }
}

struct polynomialmodq1 {
    
    var coefs: [Int] = []
    var degree: Int {
        get {
            return coefs.count - 1
        }
    }
    
    func multiplypolys(_ poly1: polynomialmodq1) -> polynomialmodq1 {  // returns the product with powers mod q
        
        var newpoly = polynomialmodq1()
        var degreetracker: Int = 0
        var multiplier: Int = 1     // to get multiples of numbers to combine powers mod q1
        
        for multindex in 0...poly1.degree + self.degree {
            
            newpoly.coefs.append(0)
            
            while degreetracker <= multindex && degreetracker <= poly1.degree {  //multiply all possible combinations add to correct degree
                
                if multindex - degreetracker <= self.degree {
                    
                    newpoly.coefs[multindex] += (poly1.coefs[degreetracker]*self.coefs[multindex - degreetracker])
                    
                }
                
                degreetracker += 1
                
            }
            
            degreetracker = 0
        }
        
        // knock powers >= q1 down and add them in
        
        if newpoly.degree >= q1 {
            
            for powerindex in 0...q1 - 1 {
                
                while multiplier*(q1) + powerindex <= newpoly.degree {
                    
                    newpoly.coefs[powerindex % q1] += newpoly.coefs[multiplier*q1 + powerindex]
                    multiplier += 1
                    
                }
                multiplier = 1
            }
            
            for powerindex in q1...newpoly.degree {      // kill everything from x^q1 on
                
                newpoly.coefs.removeLast()
                
            }
            
        }
        
        return newpoly
    }
    
    func addpolys(_ poly1: polynomialmodq1) -> polynomialmodq1 {        // returns the sum with powers mod q1
        
        var newpoly = polynomialmodq1()
        var multiplier: Int = 1     // to get multiples of numbers to combine powers mod q1
        
        for addindex in 0...max(poly1.degree,self.degree) {
            
            newpoly.coefs.append(0)
            
            if (addindex <= min(poly1.degree,self.degree)) {  // if can add both, do it, otherwise pick correct one
                
                newpoly.coefs[addindex] = (poly1.coefs[addindex] + self.coefs[addindex])
                
            } else if addindex <= poly1.degree && addindex > self.degree {
                
                newpoly.coefs[addindex] = poly1.coefs[addindex]
                
            } else if addindex > poly1.degree && addindex <= self.degree {
                
                newpoly.coefs[addindex] = self.coefs[addindex]
                
            }
        }
        
        
        // knock powers >= q1 down and add them in
        
        if newpoly.degree >= q1 {
            
            for powerindex in 0...q1 - 1{
                
                while multiplier*q1 + powerindex <= newpoly.degree {
                    
                    newpoly.coefs[powerindex % q1] += newpoly.coefs[multiplier*q1 + powerindex]
                    multiplier += 1
                    
                }
                multiplier = 1
            }
            
            for powerindex in q1...newpoly.degree {      // kill everything from x^q on
                
                newpoly.coefs.removeLast()
                
            }
            
        }
        return newpoly
    }
}

struct polynomialmodq2 {
    
    var coefs: [Int] = []
    var degree: Int {
        get {
            return coefs.count - 1
        }
    }
    
    func multiplypolys(_ poly1: polynomialmodq2) -> polynomialmodq2 {  // returns the product with powers mod q
        
        var newpoly = polynomialmodq2()
        var degreetracker: Int = 0
        var multiplier: Int = 1     // to get multiples of numbers to combine powers mod q2
        
        for multindex in 0...poly1.degree + self.degree {
            
            newpoly.coefs.append(0)
            
            while degreetracker <= multindex && degreetracker <= poly1.degree {  //multiply all possible combinations add to correct degree
                
                if multindex - degreetracker <= self.degree {
                    
                    newpoly.coefs[multindex] += (poly1.coefs[degreetracker]*self.coefs[multindex - degreetracker])
                    
                }
                
                degreetracker += 1
                
            }
            
            degreetracker = 0
        }
        
        // knock powers >= q2 down and add them in
        
        if newpoly.degree >= q2 {
            
            for powerindex in 0...q2 - 1 {
                
                while multiplier*q2 + powerindex <= newpoly.degree {
                    
                    newpoly.coefs[powerindex % q2] += newpoly.coefs[multiplier*q2 + powerindex]
                    multiplier += 1
                    
                }
                multiplier = 1
            }
            
            for powerindex in q2...newpoly.degree {      // kill everything from x^q on
                
                newpoly.coefs.removeLast()
                
            }
            
        }
        
        return newpoly
    }
    
    func addpolys(_ poly1: polynomialmodq2) -> polynomialmodq2 {        // returns the sum with powers mod q2
        
        var newpoly = polynomialmodq2()
        var multiplier: Int = 1     // to get multiples of numbers to combine powers mod q2
        
        for addindex in 0...max(poly1.degree,self.degree) {
            
            newpoly.coefs.append(0)
            
            if (addindex <= min(poly1.degree,self.degree)) {  // if can add both, do it, otherwise pick correct one
                
                newpoly.coefs[addindex] = (poly1.coefs[addindex] + self.coefs[addindex])
                
            } else if addindex <= poly1.degree && addindex > self.degree {
                
                newpoly.coefs[addindex] = poly1.coefs[addindex]
                
            } else if addindex > poly1.degree && addindex <= self.degree {
                
                newpoly.coefs[addindex] = self.coefs[addindex]
                
            }
        }
        
        
        // knock powers >= q2 down and add them in
        
        if newpoly.degree >= q2 {
            
            for powerindex in 0...q2 - 1 {
                
                while multiplier*q2 + powerindex <= newpoly.degree {
                    
                    newpoly.coefs[powerindex % q2] += newpoly.coefs[multiplier*q2 + powerindex]
                    multiplier += 1
                    
                }
                multiplier = 1
            }
            
            for powerindex in q2...newpoly.degree {      // kill everything from x^q on
                
                newpoly.coefs.removeLast()
                
            }
            
        }
        return newpoly
    }
}

var poly1 = polynomialmodq(coefs: [1])      // the poly = 1
var poly0 = polynomialmodq(coefs: [0])      // the poly = 0
var poly11 = polynomialmodq1(coefs: [1])      // the poly = 1
var poly01 = polynomialmodq1(coefs: [0])      // the poly = 0
var poly12 = polynomialmodq2(coefs: [1])      // the poly = 1
var poly02 = polynomialmodq2(coefs: [0])      // the poly = 0

struct doublepolynomial { // polynomials with polynomial coefficients, powers mod q
    
    var coefs: [polynomialmodq] = []
    var degree: Int {
        get {
            return coefs.count - 1
        }
    }
    
    func multiplydoublepolys(_ poly: doublepolynomial) -> doublepolynomial {
        
        var newpoly = doublepolynomial()
        var degreetracker: Int = 0
        let zeropoly = polynomialmodq(coefs: [0])
        
        for multindex in 0...poly.degree + self.degree {
            
            newpoly.coefs.append(zeropoly)
            
            while degreetracker <= multindex && degreetracker <= poly.degree {
                
                if multindex - degreetracker <= self.degree {
                    
                    newpoly.coefs[multindex] = newpoly.coefs[multindex].addpolys(poly.coefs[degreetracker].multiplypolys(self.coefs[multindex - degreetracker]))       // add all possible products for each degree
                    
                }
                
                degreetracker += 1
                
            }
            
            degreetracker = 0
        }
        
        return newpoly
    }
    
    func multiplybybinomial(_ xpower: Int) -> doublepolynomial {  //  FIX, need to foil for y coefs
        
        // multiply this double by 1+x^n*y, for speed
        
        var arraytopadwith: [Int] = []   // for padding front of poly to multiply by x^n
        var bufferpoly = polynomialmodq()   //for holding a copy of the coef polynomial which can be padded
        var newdoublepoly = doublepolynomial()
        
        for powerindex in 1...xpower {  // set up array to multiply by x^powerindex
            
            arraytopadwith.append(0)
            
        }
        
        newdoublepoly.coefs.append(self.coefs[0])      //first coef doesn't change
        
        if self.degree != 0 {for polyindex in 1...self.degree {
            
            bufferpoly = self.coefs[polyindex - 1]      //new coef is current one plus x^n times previous one
            bufferpoly.coefs.insert(contentsOf: arraytopadwith, at: 0)
            newdoublepoly.coefs.append(bufferpoly.addpolys(self.coefs[polyindex]))
            
            }}
        
        // and finally add the last coef
        
        bufferpoly = self.coefs[self.degree]
        bufferpoly.coefs.insert(contentsOf: arraytopadwith, at: 0)
        newdoublepoly.coefs.append(bufferpoly.addpolys(poly0))      // knock x powers down mod q
        
        return newdoublepoly
        
    }
}

struct doublepolynomialq1 { // polynomials with polynomial coefficients, powers mod q1
    
    var coefs: [polynomialmodq1] = []
    var degree: Int {
        get {
            return coefs.count - 1
        }
    }
    
    func multiplydoublepolys(_ poly: doublepolynomialq1) -> doublepolynomialq1 {
        
        var newpoly = doublepolynomialq1()
        var degreetracker: Int = 0
        let zeropoly = polynomialmodq1(coefs: [0])
        
        for multindex in 0...poly.degree + self.degree {
            
            newpoly.coefs.append(zeropoly)
            
            while degreetracker <= multindex && degreetracker <= poly.degree {
                
                if multindex - degreetracker <= self.degree {
                    
                    newpoly.coefs[multindex] = newpoly.coefs[multindex].addpolys(poly.coefs[degreetracker].multiplypolys(self.coefs[multindex - degreetracker]))       // add all possible products for each degree
                    
                }
                
                degreetracker += 1
                
            }
            
            degreetracker = 0
        }
        
        return newpoly
    }
    
    func multiplybybinomial(_ xpower: Int) -> doublepolynomialq1 {  //  FIX, need to foil for y coefs
        
        // multiply this double by 1+x^n*y, for speed
        
        var arraytopadwith: [Int] = []   // for padding front of poly to multiply by x^n
        var bufferpoly = polynomialmodq1()   //for holding a copy of the coef polynomial which can be padded
        var newdoublepoly = doublepolynomialq1()
        
        for powerindex in 1...xpower {  // set up array to multiply by x^powerindex
            
            arraytopadwith.append(0)
            
        }
        
        newdoublepoly.coefs.append(self.coefs[0])      //first coef doesn't change
        
        if self.degree != 0 {for polyindex in 1...self.degree {
            
            bufferpoly = self.coefs[polyindex - 1]      //new coef is current one plus x^n times previous one
            bufferpoly.coefs.insert(contentsOf: arraytopadwith, at: 0)
            newdoublepoly.coefs.append(bufferpoly.addpolys(self.coefs[polyindex]))
            
            }}
        
        // and finally add the last coef
        
        bufferpoly = self.coefs[self.degree]
        bufferpoly.coefs.insert(contentsOf: arraytopadwith, at: 0)
        newdoublepoly.coefs.append(bufferpoly.addpolys(poly01))      // knock x powers down mod q
        
        return newdoublepoly
        
    }
}

struct doublepolynomialq2 { // polynomials with polynomial coefficients, powers mod q2
    
    var coefs: [polynomialmodq2] = []
    var degree: Int {
        get {
            return coefs.count - 1
        }
    }
    
    func multiplydoublepolys(_ poly: doublepolynomialq2) -> doublepolynomialq2 {
        
        var newpoly = doublepolynomialq2()
        var degreetracker: Int = 0
        let zeropoly = polynomialmodq2(coefs: [0])
        
        for multindex in 0...poly.degree + self.degree {
            
            newpoly.coefs.append(zeropoly)
            
            while degreetracker <= multindex && degreetracker <= poly.degree {
                
                if multindex - degreetracker <= self.degree {
                    
                    newpoly.coefs[multindex] = newpoly.coefs[multindex].addpolys(poly.coefs[degreetracker].multiplypolys(self.coefs[multindex - degreetracker]))       // add all possible products for each degree
                    
                }
                
                degreetracker += 1
                
            }
            
            degreetracker = 0
        }
        
        return newpoly
    }
    
    func multiplybybinomial(_ xpower: Int) -> doublepolynomialq2 {  //  FIX, need to foil for y coefs
        
        // multiply this double by 1+x^n*y, for speed
        
        var arraytopadwith: [Int] = []   // for padding front of poly to multiply by x^n
        var bufferpoly = polynomialmodq2()   //for holding a copy of the coef polynomial which can be padded
        var newdoublepoly = doublepolynomialq2()
        
        for powerindex in 1...xpower {  // set up array to multiply by x^powerindex
            
            arraytopadwith.append(0)
            
        }
        
        newdoublepoly.coefs.append(self.coefs[0])      //first coef doesn't change
        
        if self.degree != 0 {for polyindex in 1...self.degree {
            
            bufferpoly = self.coefs[polyindex - 1]      //new coef is current one plus x^n times previous one
            bufferpoly.coefs.insert(contentsOf: arraytopadwith, at: 0)
            newdoublepoly.coefs.append(bufferpoly.addpolys(self.coefs[polyindex]))
            
            }}
        
        // and finally add the last coef
        
        bufferpoly = self.coefs[self.degree]
        bufferpoly.coefs.insert(contentsOf: arraytopadwith, at: 0)
        newdoublepoly.coefs.append(bufferpoly.addpolys(poly02))      // knock x powers down mod q
        
        return newdoublepoly
        
    }
}


struct etapart {
    
    var powers: [Int]
    var coefs: [Int]
    
    func addetaparts(_ toaddto: etapart) -> etapart {
        
        var newetaparts = etapart(powers: [], coefs: [])
        var indexofmatching: Int?
        
        for firstpowerindex in 0...self.powers.count - 1 {
            
            newetaparts.powers.append(self.powers[firstpowerindex])     // add first power to powers and first coef to coefs
            newetaparts.coefs.append(self.coefs[firstpowerindex])
            
        }
        
        
        for secondpowerindex in 0...toaddto.powers.count - 1 {
            
            if newetaparts.powers.contains(toaddto.powers[secondpowerindex]) {  // if second power already there
                
                indexofmatching = newetaparts.powers.index(of: toaddto.powers[secondpowerindex])
                
                newetaparts.coefs[indexofmatching!] += toaddto.coefs[secondpowerindex] // add coef to existing one
                
            } else {
                
                newetaparts.powers.append(toaddto.powers[secondpowerindex])     // otherwise just append it with coef
                newetaparts.coefs.append(toaddto.coefs[secondpowerindex])
                
            }
            
        }
        
        // sort powers and coefs based on powers
        
        let combined = zip(newetaparts.powers, newetaparts.coefs).sorted {$0.0 < $1.0}
        
        newetaparts.powers = combined.map {$0.0}
        newetaparts.coefs = combined.map {$0.1}
        
        return newetaparts
        
    }
}



var nextschoice: Set<Int> = [1]    //next array to add to schoices
var subsetlistindex: Int = 0
var subsetindex: Int = 0

var mult: Int = 2
var currentschoice: Set<Int>
var currentschoicedup: Set<Int>
var donemultiplying: Bool = false

var complete: Set<Int> = [1,q-1]
var s: Set<Int>
var r: Set<Int>
var S = Array<Array<Int>>()
var R = Array<Array<Int>>()
var S1 = Array<Array<Int>>()
var R1 = Array<Array<Int>>()
var S2 = Array<Array<Int>>()
var R2 = Array<Array<Int>>()
S = Array(repeating: Array(repeating: 0, count: q), count: 2*k + 1)
R = Array(repeating: Array(repeating: 0, count: 2*n + 1), count: q)
S1 = Array(repeating: Array(repeating: 0, count: q2), count: 2*k + 1)
R1 = Array(repeating: Array(repeating: 0, count: 2*n + 1), count: q2)
S2 = Array(repeating: Array(repeating: 0, count: q1), count: 2*k + 1)
R2 = Array(repeating: Array(repeating: 0, count: 2*n + 1), count: q1)
var choices: Array<Array<Int>>
var pick:Int
var Gs = doublepolynomial(coefs: [poly1])
var Gr = doublepolynomial(coefs: [poly1])
var Gs1 = doublepolynomialq1(coefs: [poly11])
var Gr1 = doublepolynomialq1(coefs: [poly11])
var Gs2 = doublepolynomialq2(coefs: [poly12])
var Gr2 = doublepolynomialq2(coefs: [poly12])
var p = polynomialmodq()    //for picking off the coeffiecient polynomials from Gs and Gr
var p1 = polynomialmodq1()    //for picking off the coeffiecient polynomials from Gs and Gr
var p2 = polynomialmodq2()    //for picking off the coeffiecient polynomials from Gs and Gr

var sizes: Int = 2*k
var sizer: Int = 2*n

var C = Array<Array<Int>>()
C = Array(repeating: Array(repeating: 0, count: sizer + 1), count: sizes + 1)
var C1 = Array<Array<Int>>()
C1 = Array(repeating: Array(repeating: 0, count: sizer + 1), count: sizes + 1)
var C2 = Array<Array<Int>>()
C2 = Array(repeating: Array(repeating: 0, count: sizer + 1), count: sizes + 1)
var rowtimescolumntotal: Int = 0
var sindex: Int = 0
var newetapart = polynomial(coefs: [])

var rowcompare: [Bool]
var matchlist: [[Int]] = []
var inequalrun: Bool = false        // set to true if I'm in a run of equal etaparts
var howmanyruns: Int = 0

func subsets(_ source: [Int], takenBy : Int) -> [[Int]] {
    if(source.count == takenBy) {
        return [source]
    }
    
    if(source.isEmpty) {
        return []
    }
    
    if(takenBy == 0) {
        return []
    }
    
    if(takenBy == 1) {
        return source.map { [$0] }
    }
    
    var result : [[Int]] = []
    
    let rest = Array(source.suffix(from: 1))
    let sub_combos = subsets(rest, takenBy: takenBy - 1)
    result += sub_combos.map { [source[0]] + $0 }
    
    result += subsets(rest, takenBy: takenBy)
    
    return result
}

func multsetby(_ settomult: Set<Int>,multby: Int) -> Set<Int> {
    var setelement: Int = 0
    var localsettomult: Set<Int> = []
    
    for setelement in settomult {
        localsettomult.insert(multby*setelement % q)
    }
    
    return localsettomult
}
func totalset(_ settoadd: Array<Int>) -> Int {
    var total: Int = 0
    
    for number in settoadd {
        total += number
    }
    
    return total
}

func countoccurrences(_ inarray: [Int],_ whattofind: Int) -> Int {
    var currentcount: Int = 0
    
    for idx in inarray {
        if idx == whattofind {currentcount += 1}
    }
    
    return currentcount
}

func powint(_ base: Int,power: Int) -> Int {
    
    var answer: Int = 1
    
    if power != 0 {
        for _ in 1...power {
            
            answer *= base
            
        }
    }
    
    return answer
}

func compareetarows(_ etatocheck: Array<Array<polynomial>>, row1: Int, row2: Int) -> [Bool] {  // check two rows of eta for equal values, return true if equal for given p (column)
    
    var part1: polynomial
    var part2: polynomial
    var comparelist: [Bool] = []
    
    for pindex in 0...n {
        
        part1 = etatocheck[row1][pindex]
        part2 = etatocheck[row2][pindex]
        
        if part1.coefs == part2.coefs {comparelist.append(true)} else {comparelist.append(false)}
        
    }
    
    return comparelist
}

relprime = []

while i <= q/2 {
    if GCD(i,q) == 1 {relprime.append(i)}
    i += 1
}


subsetlist = subsets(relprime,takenBy:k-1)
subsetlist.count

while subsetlistindex < subsetlist.count {
    while subsetindex < subsetlist[subsetlistindex].count {
        nextschoice.insert(subsetlist[subsetlistindex][subsetindex])
        nextschoice.insert(q - subsetlist[subsetlistindex][subsetindex])
        subsetindex += 1
    }
    nextschoice.insert(q - 1)
    schoices.insert(nextschoice)
    nextschoice = [1]
    subsetindex = 0
    subsetlistindex += 1
}
print("\(q) , \(schoices.count)\n")



for currentschoice in schoices {
    
    if schoices.contains(currentschoice) {
        
        while mult < q/2 && !donemultiplying {
            
            currentschoicedup = multsetby(currentschoice, multby: mult)
            if schoices.contains(currentschoicedup) {
                schoices.remove(currentschoicedup)
                donemultiplying = true
            }
            mult += 1
        }
        mult = 2
        donemultiplying = false
    }
}


print("\(schoices.count)")


for relprimeint in relprime {
    complete.insert(relprimeint)
    complete.insert(q-relprimeint)
}

var eta = Array(repeating: Array(repeating: polynomial(coefs: [0]), count: n + 1), count: schoices.count)

for s in schoices {  // loop through valid k-tuples
    
    r = complete.subtracting(s)
    
    for sint in s {
        
        Gs = Gs.multiplybybinomial(sint)
        Gs1 = Gs1.multiplybybinomial(sint)
        Gs2 = Gs2.multiplybybinomial(sint)
        
    }
    
    for rint in r {
        
        Gr = Gr.multiplybybinomial(rint)
        Gr1 = Gr1.multiplybybinomial(rint)
        Gr2 = Gr2.multiplybybinomial(rint)
        
    }
    
    for pick in 0...sizes {
        
        if Gs.coefs.count > pick {p = Gs.coefs[pick]} else {p = poly0}
        for sidx in 1...q {
            if p.coefs.count >= sidx {
                S[pick][sidx - 1] = p.coefs[sidx - 1]
            }
        }
    }
    
    for pick in 0...sizer {
        if Gr.coefs.count > pick {p = Gr.coefs[pick]} else {p = poly0}
        R[0][pick] = p.coefs[0]     //first R is the coef of x^0
        for ridx in 1...(q - 1) {   //the rest load backwards
            if p.coefs.count > ridx {
                R[q - ridx][pick] = p.coefs[ridx % q]
            }
        }
    }
    
    
    for pick in 0...sizes {
        
        if Gs1.coefs.count > pick {p1 = Gs1.coefs[pick]} else {p1 = poly01}
        for sidx in 1...q2 {
            if p1.coefs.count >= sidx {
                S1[pick][sidx - 1] = p1.coefs[sidx - 1]
            }
        }
    }
    
    for pick in 0...sizer {
        if Gr1.coefs.count > pick {p1 = Gr1.coefs[pick]} else {p1 = poly01}
        R[0][pick] = p1.coefs[0]     //first R is the coef of x^0
        for ridx in 1...(q2 - 1) {   //the rest load backwards
            if p1.coefs.count > ridx {
                R1[q2 - ridx][pick] = p1.coefs[ridx % q2]
            }
        }
    }
    
    for pick in 0...sizes {
        
        if Gs2.coefs.count > pick {p2 = Gs2.coefs[pick]} else {p2 = poly02}
        for sidx in 1...q1 {
            if p2.coefs.count >= sidx {
                S2[pick][sidx - 1] = p2.coefs[sidx - 1]
            }
        }
    }
    
    for pick in 0...sizer {
        if Gr2.coefs.count > pick {p2 = Gr2.coefs[pick]} else {p2 = poly02}
        R[0][pick] = p2.coefs[0]     //first R is the coef of x^0
        for ridx in 1...(q1 - 1) {   //the rest load backwards
            if p2.coefs.count > ridx {
                R2[q1 - ridx][pick] = p2.coefs[ridx % q1]
            }
        }
    }
    
    
    
    
    Gs.coefs = [poly1]
    Gr.coefs = [poly1]
    
    for rowinfirst in 0...sizes {
        for columninsecond in 0...sizer {
            
            rowtimescolumntotal = 0
            
            for j in 0...q-1 {
                
                rowtimescolumntotal += S[rowinfirst][j] * R[j][columninsecond]
            }
            
            C[rowinfirst][columninsecond] =  EulerPhiq * rowtimescolumntotal
            
        }
    }
    
    for rowinfirst in 0...sizes {
        for columninsecond in 0...sizer {
            
            rowtimescolumntotal = 0
            
            for j in 0...q2-1 {
                
                rowtimescolumntotal += S1[rowinfirst][j] * R1[j][columninsecond]
            }
            
            C1[rowinfirst][columninsecond] =  q2 * rowtimescolumntotal
            
        }
    }
    
    for rowinfirst in 0...sizes {
        for columninsecond in 0...sizer {
            
            rowtimescolumntotal = 0
            
            for j in 0...q1-1 {
                
                rowtimescolumntotal += S2[rowinfirst][j] * R2[j][columninsecond]
            }
            
            C2[rowinfirst][columninsecond] =  q1 * rowtimescolumntotal
            
        }
    }
    
    
    
    for p in 0...n {
        for a in 0...2*k {
            for t in 0...p {
                
                newetapart = polynomial(coefs:Cyclotomicq1q2.coefs.map {$0 * C[a][p-t]})
                eta[sindex][p] = eta[sindex][p].addpolys(newetapart)
                newetapart = polynomial(coefs:Cyclotomicq1q.coefs.map {$0 * C[a][p-t]})
                eta[sindex][p] = eta[sindex][p].addpolys(newetapart)
                newetapart = polynomial(coefs:Cyclotomicqq2.coefs.map {$0 * C[a][p-t]})
                eta[sindex][p] = eta[sindex][p].addpolys(newetapart)
                
            }
        }
    }
    
    schoicesarray.append(s)     // build array with k tuples in same order for later
    sindex += 1
    print(sindex)
}

var runs: Array<Int> = []
var runindex: Int = 0
var nonzerocount: Int = 0


for i in 0...schoices.count - 2 {
    
    for j in i + 1...schoices.count - 1 {
        
        rowcompare = compareetarows(eta, row1: i, row2: j)
        
        for kk in 0...rowcompare.count - 1 {
            
            if !inequalrun && rowcompare[kk] {
                
                matchlist.append([kk,0])
                howmanyruns += 1
                inequalrun = true
                
            } else if inequalrun && !rowcompare[kk] {
                
                matchlist[howmanyruns - 1][1] = kk - 1
                inequalrun = false
                
            }
        }
        
        while runindex < matchlist.count {
            
            runs.append(matchlist[runindex][1] - matchlist[runindex][0])
            if runs[runindex] > 0 {nonzerocount += 1}
            runindex += 1
            
            
        }
        
        if nonzerocount > 1 || (nonzerocount == 1 && matchlist[0][0] > 0) {print(matchlist, schoicesarray[i],schoicesarray[j])}
        
        howmanyruns = 0
        matchlist = []
        inequalrun = false
        runs = []
        runindex = 0
        nonzerocount = 0
    }
}
\end{verbatim}
}

On behalf of all authors, the corresponding author states that there is no conflict of interest. 

\bibliographystyle{amsplain}

\end{document}